\documentclass[11pt]{amsart}
\usepackage{amssymb,latexsym,amsmath,graphicx,graphics,epic,eepic}

\theoremstyle{plain}
\newtheorem{theorem}{Theorem}
\numberwithin{theorem}{section}
\newtheorem{lemma}[theorem]{Lemma}
\newtheorem{proposition}[theorem]{Proposition}

\theoremstyle{definition}

\newtheorem{remark}[theorem]{Remark}

\newcommand{\C}{{\mathbb C}}
\newcommand{\R}{{\mathbb R}}
\newcommand{\Z}{{\mathbb Z}}

\newcommand{\s}{{\mathbb S}}

\newcommand{\I}{{\mathbb I}}

\begin{document}
\title{Symmetric symplectic homotopy $K3$ surfaces}
\author{Weimin Chen and Slawomir Kwasik}
\thanks{The first author was partially supported by NSF grant DMS-0603932}
\thanks{The second author was partially supported by BOR award LEQSF(2008-2011)-RD-A-24}
\date{}

\begin{abstract}
A study on the relation between the smooth structure of a symplectic homotopy $K3$ surface 
and its symplectic symmetries is initiated. A measurement of exoticness of a symplectic homotopy 
$K3$ surface is introduced, and the influence of an effective action of a $K3$ group via
symplectic symmetries is investigated. It is shown that an effective action by various maximal
symplectic $K3$ groups forces the corresponding homotopy $K3$ surface to be minimally exotic 
with respect to our measure. (However, the standard $K3$ is the only known example of such 
minimally exotic homotopy $K3$ surfaces.)
The possible structure of a finite group of symplectic symmetries of a minimally exotic homotopy 
$K3$ surface is determined and future research directions are indicated.
\end{abstract}

\maketitle

\section{Introduction}

In the recent advances in topology and geometry of smooth $4$-manifolds
a very important role was played by one particular class of $4$-manifolds,
namely, the homotopy $K3$ surfaces. These manifolds have been used to
test the flexibility of smooth and symplectic structures in comparison
with the rigidity of holomorphic structures. To be more precise, let
$X$ be a homotopy $K3$ surface, namely, $X$ is a closed, oriented smooth
$4$-manifold homeomorphic (as an oriented manifold) to the standard $K3$
surface. If such a manifold admits an orientation-compatible symplectic
structure, then it is called a symplectic homotopy $K3$ surface. While
the knot surgery of Fintushel and Stern (cf. \cite{FS}) allows construction
of numerous examples of symplectic homotopy $K3$ surfaces, deep work
of Taubes \cite{T} gives very strong information about the
smooth structures on such manifolds. For example, one can easily show that
the set of Seiberg-Witten basic classes of $X$ spans an isotropic sublattice
$L_X$ of $H^2(X;\Z)$ (with respect to the cup product), so that its rank,
denoted by $r_X$, must range from $0$ to $3$ (cf. Proposition 4.1).
The rank $r_X$ of the lattice $L_X$ of the Seiberg-Witten basic classes
gives a rough measurement of the exoticness of the smooth structure of $X$,
with $r_X=0$ being the minimally exotic and with $r_X=3$ being the maximally exotic.

There are various known characterizations of the minimally exotic (i.e. $r_X=0$)
symplectic homotopy $K3$ surfaces $X$, which are all characteristics of the
standard $K3$, namely:

\begin{itemize}
\item $X$ has a trivial canonical class, i.e., $c_1(K_X)=0$, cf. \cite{T}.
\item $X$ has a unique Seiberg-Witten basic class, cf. \cite{MS, T}.
\item $X$ has the same Seiberg-Witten invariant of the standard $K3$, cf.
          \cite{T}.
\item $X$ is a simply-connected, minimal symplectic $4$-manifold with zero
       Kodaira dimension, cf. \cite{Li}.
\end{itemize}
Moreover, the standard $K3$ surface is the only known example of such a $4$-manifold,
and it has been a challenging problem as whether there is an exotic smooth structure with
$r_X=0$. 

\vspace{2mm}

In \cite{CK2} the authors studied the possible effect of a change of
a smooth structure on the symmetry group of a closed, oriented
$4$-manifold. It was shown that for an infinite family of maximally exotic
(i.e. $r_X=3$) $K3$ surfaces, there are very significant
constraints on the smooth symmetry groups of the manifolds.  
The current paper took a rather opposite viewpoint as one of its purposes 
is to investigate the implications of a (symplectic) group action for
the smooth structure of a $4$-manifold.

The interaction between smooth structures and symmetry groups of a
manifold is one of the basic questions in the theory of differentiable
transformation groups. In particular, the following classical theorem
of differential geometry gives a characterization of the standard sphere
$\s^n$ among all the homotopy $n$-spheres as having the largest degree of
symmetry (cf. \cite{Hsiang}).

\vspace{2mm}

{\bf Theorem} (A Characterization of $\s^n$).
{\em Let $M^n$ be a closed, simply connected
manifold of dimension $n$, and let $G$ be a compact Lie group which
acts smoothly and effectively on $M^n$. Then $\dim G\leq n(n+1)/2$,
with equality if and only if $M^n$ is diffeomorphic to $\s^n$.}

\vspace{2mm}

If $X$ is a homotopy $K3$ surface then it is well known that a compact
Lie group acting smoothly on $X$ must be finite (cf. \cite{AH}).
A finite group $G$ is called a $K3$ group (resp. symplectic $K3$ group)
if $G$ can be realized as a subgroup of the automorphism group (resp.
symplectic automorphism group) of a  $K3$ surface. Finite automorphism
groups of $K3$ surfaces were first systematically studied by Nikulin in
\cite{N}; in particular, he completely classified finite abelian groups
of symplectic automorphisms. Subsequently, Mukai \cite{Mu} determined all
the symplectic $K3$ groups (see also \cite{Kon, Xiao}). The following $11$ groups 
are the maximal symplectic $K3$ groups:
$$
L_2(7), A_6, S_5, M_{20}, F_{384}, A_{4,4}, T_{192}, H_{192}, N_{72},
M_9, T_{48}.
$$

Motivated by the above characterization of the standard $\s^n$ we were led to the following:

\vspace{2mm}

{\bf Problem} \hspace{2mm} {\it Let $X$ be a homotopy $K3$ surface supporting
an effective action of a ``large'' $K3$ group via symplectic symmetries.
What can be said about the smooth structure on $X$?}

\vspace{2mm}

Viewing the above maximal symplectic $K3$ groups as ``large'', our
solution to this problem is contained in the following:

\vspace{3mm}

\noindent{\bf Theorem 1.0.} {\em Let $G$ be one of the following maximal
symplectic $K3$ groups: 
$$
L_2(7), A_6, M_{20}, A_{4,4}, T_{192}, T_{48},
$$ 
and let $X$ be a symplectic homotopy $K3$ surface. If $X$ admits an
effective $G$-action via symplectic symmetries, then $X$ must
be minimally exotic, i.e., $r_X=0$.
}

\vspace{3mm}

{\bf Remarks} \hspace{3mm}
It is possible to extend Theorem 1.0 to other $K3$ groups, or more generally, to give 
an upper bound on the exoticness $r_X$ when $X$ admits a ``relatively large'' symplectic 
symmetry group. However, we shall not pursue these extensions here as the detailed analysis
depends very much on the structure of each individual group involved.

\vspace{3mm}

In fact, behind the proof of Theorem 1.0 a general method was devised in this paper
which allows one to measure the effect of a symplectic finite group action on a homotopy
$K3$ surface $X$ in terms of its exoticness $r_X$. 
The basic idea of our method may be summarized as follows. Let a finite
group $G$ act on a homotopy $K3$ surface $X$ via symplectic symmetries.
Using the techniques developed in our previous work \cite{CK1} and exploiting
various features of the structure of $G$, one first determines the possible
fixed point set of an arbitrary element $g\in G$, from which the trace $tr(g)$ of
$g$ on $H^\ast(X;\Z)$ can be computed using the Lefschetz fixed point theorem.
This leads to a calculation of
$$
\dim (H^\ast(X;\R))^G=\frac{1}{|G|}\sum_{g\in G}tr(g).
$$
On the other hand, there is an induced action of $G$ on the lattice $L_X$ of
the Seiberg-Witten basic classes. The following basic inequality 
$$
\dim (L_X\otimes_\Z\R)^G\leq \min(b^{+}_2(X/G),b^{-}_2(X/G)),
$$
which follows from the fact that $L_X$ is isotropic (cf. Proposition 4.1), 
plus the identity $\dim (H^\ast(X;\R))^G=2+b^{+}_2(X/G)+b^{-}_2(X/G)$
allows one to obtain information about $\dim (L_X\otimes_\Z\R)^G$
and $r_X=\text{rank }L_X$.

For an illustration we consider the case where $G$ is a nonabelian simple
group. It is easily seen that in this case $b^{+}_2(X/G)=3$ and
$r_X=\dim (L_X\otimes_\Z\R)^G$. The above basic inequality then becomes
$$
r_X\leq \min(3, \dim (H^\ast(X;\R))^G-5).
$$
There are three nonabelian simple $K3$ groups: $L_2(7), A_5$ and $A_6$.
For the case where $G=L_2(7)$ or $A_6$, we show in Section 2 that
$\dim (H^\ast(X;\R))^G=5$ (which is the same as that of a holomorphic
$G$-action on a $K3$ surface), so that $r_X=0$ as asserted in Theorem 1.0.
For $G=A_5$, the fixed-point analysis only gives $\dim (H^\ast(X;\R))^G\leq 8$, cf. Lemma 2.5.
(Note that even for a holomorphic $A_5$-action, one only gets 
$\dim (H^\ast(X;\R))^G=6$.) Thus in the case of $G=A_5$, our method only gives $r_X\leq 3$, 
which does not yield any restriction on the exoticness $r_X$.

\vspace{2mm}

Our Theorem 1.0 naturally gives rise to the following question.

\vspace{2mm}

{\it What can be said about a finite group $G$ which can act effectively
on a minimally exotic symplectic homotopy $K3$ surface via symplectic
symmetries?
}

\vspace{2mm}

In the following theorem, we show that the symmetries of a minimally exotic
symplectic homotopy $K3$ surface look very much
like holomorphic automorphisms of the standard $K3$ surface; in particular,
the symmetry groups are more or less $K3$ groups.

\begin{theorem}\label{th:thm1.1}
Let $X$ be a minimally exotic symplectic homotopy $K3$ surface (i.e. $r_X=0$)
and let $G$ be a finite group acting effectively on $X$ via symplectic symmetries. Then
there exists a short exact sequence of finite groups
$$
1\rightarrow G_0\rightarrow G\rightarrow G^0\rightarrow 1,
$$
where $G^0$ is cyclic and $G_0$ is a symplectic $K3$ group, such that
$G_0$ is characterized as the maximal subgroup of $G$ with the property
$b_2^{+}(X/G_0)=3$. Moreover, the induced action of $G_0$ on $X$ has the
same fixed point set structure as does a symplectic holomorphic action on the standard
$K3$ by $G_0$.
\end{theorem}

Motivated by the above result we turn our attention to the problem of constructing
symplectic finite group actions on {\it exotic} $K3$ surfaces.  It seems that this line 
of research will require a developement of new techniques. Our next result could 
be viewed as a first, preliminary, step in this direction.

First of all, it is clear that the Fintushel-Stern knot surgery \cite{FS} can be suitably
adapted for this purpose. More precisely, suppose a finite group $G$ acts on the 
standard $K3$ surface preserving an elliptic fibration. Then under a certain condition
(cf. Reamrk 4.3), one can perform knot surgery equivariantly to produce $G$-actions
on exotic $K3$ surfaces. For example, every cyclic $K3$ group of prime order can act 
holomorphically on an elliptic $K3$ surface (cf. \cite{Shi, MO}), and by a knot surgery 
one can easily show that such a group can act on an exotic $K3$ surface via symplectic 
symmetries. Concerning noncyclic $K3$ groups, the following theorem perhaps gives the
most dramatic example of such a construction.

\begin{theorem}\label{th:thm1.2}
Let $G\equiv (\Z_2)^3$.
There exists an infinite family of distinct maximally exotic (i.e. $r_X=3$) symplectic 
homotopy $K3$ surfaces, such that each member
of the exotic $K3$'s admits an effective $G$-action via symplectic symmetries.
Moreover, the $G$-action is pseudofree and induces a trivial action on the
lattice $L_X$ of the Seiberg-Witten basic classes.
\end{theorem}

The limitation of equivariant knot surgery is that the group $G$ has to be a $K3$ group,
and that it is difficult to construct group actions on homotopy $K3$ surfaces with a large 
exoticness (e.g., $r_X>1$). In particular, the following questions seem to require techniques 
which go beyond the equivariant knot surgery:

\vspace{3mm}

{\bf Questions} {\it

(1) Are there any finite groups other than a $K3$ group which can act 
symplectically on a homotopy $K3$ surface?

(2) Are there any finite groups other than $(\Z_2)^3$ (or a subgroup of it)
which can act symplectically on a homotopy $K3$ surface $X$ with $r_X=3$?

(3) Can $(\Z_2)^4$ act symplectically on an exotic $K3$ surface (i.e., with $r_X>0$)?
}
\vspace{3mm}

We would like to point out that an earlier version of this paper was circulated under the title: {\em Symmetric Homotopy $K3$ Surfaces}.

The organization of the rest of the paper is as follows. The proofs of
Theorem 1.0 and Theorem 1.1 are given in Sections 2 and 3 respectively.
In Section 4 we show that the lattice $L_X$ of Seiberg-Witten basic classes
is isotropic and $r_X\leq 3$. The proof of Theorem 1.2 is also given in
Section 4.

\section{Proof of Theorem 1.0}

Let $(X,\omega)$ be a symplectic homotopy $K3$ surface, and let $G$ be a
finite group which acts on $X$ smoothly and effectively, preserving the symplectic structure 
$\omega$. We pick an arbitrary $\omega$-compatible, $G$-equivariant almost complex
structure $J$ on $X$, and we denote by $g_J$ the associated Riemannian metric,
i.e., $g_J(\cdot,\cdot)\equiv \omega(\cdot,J\cdot)$, which is also $G$-equivariant.

We derive some preliminary information about the $G$-action first.

\begin{lemma}
Let $G_0$ be the maximal subgroup of $G$ such that $b_2^{+}(X/G_0)=3$.
Then $G/G_0$ is cyclic. Moreover, the commutator $[G,G]$ is
contained in $G_0$.
\end{lemma}

\begin{proof}
Let $H^{+}$ be the space of $g_J$-self-dual harmonic $2$-forms on $X$.
Since the Riemannian metric $g_J$ is $G$-equivariant, we see that $H^{+}$
is invariant under the action of $G$. Moreover, since $\omega\in H^{+}$
and $G$ fixes $\omega$, there is an induced action of $G$ on the
orthogonal complement $\langle\omega\rangle^{\perp}$ of $\omega$ in
$H^{+}$. Note that $\dim H^{+}=3$, so that
$\dim\;\langle\omega\rangle^{\perp}=2$. We claim that the action of $G$ on
$\langle\omega\rangle^{\perp}$ is orientation-preserving (i.e. there are
no reflections). 

To see this, suppose there is a $g\in G$ such that the action of $g$ on $\langle\omega\rangle^{\perp}$
is not orientation-preserving. This happens exactly when $g$ fixes a $1$-dimensional subspace, and
it follows easily that in this case $b_2^{+}(X/\langle g\rangle)=2$. On the other hand, 
$b_2^{+}(X/\langle g\rangle)$ must be odd.  This is because for the symplectic $4$-orbifold
$X/\langle g\rangle$, the dimension of the Seiberg-Witten moduli space associated to the canonical 
$Spin^\C$ structure equals $0$ (cf. \cite{C}, Appendix A). This gives rise to the
equation
$$
2\cdot \text{index of Dirac operator} + (b_1(X/\langle g\rangle)-1-b_2^{+}(X/\langle g\rangle))=0.
$$
It follows easily that  $b_2^{+}(X/\langle g\rangle)$ is odd because $b_1(X/\langle g \rangle)=0$.

With the preceding understood, we obtain an exact sequence of groups
$$
1\rightarrow G_0\rightarrow G\rightarrow \s^1,
$$
where the last homomorphism $G\rightarrow \s^1$ is induced from the action of
$G$ on $\langle\omega\rangle^{\perp}$. The lemma follows immediately from this.

\end{proof}

For a symplectic $K3$ group $G$, the commutator $[G,G]$ and the quotient group 
(i.e. the abelianization) $G/[G,G]$ is determined in \cite{Xiao}. For the purpose of later 
discussions, the list of $G$ where $G$ is maximal is reproduced below.

\begin{itemize}
\item $G=L_2(7)$: $[G,G]=G$ and $G/[G,G]=0$.
\item $G=A_6$:  $[G,G]=G$ and $G/[G,G]=0$.
\item $G=S_5$: $[G,G]=A_5$ and $G/[G,G]=\Z_2$.
\item $G=M_{20}=2^4 A_5$: $[G,G]=G$ and $G/[G,G]=0$.
\item $G=F_{384}=4^2 S_4$: $[G,G]=4^2 A_4$ and $G/[G,G]=\Z_2$.
\item $G=A_{4,4}=2^4 A_{3,3}$: $[G,G]=A_4^2$ and $G/[G,G]=\Z_2$.
\item $G=T_{192}=(Q_8\ast Q_8)\times_\phi S_3$:
$[G,G]=(Q_8\ast Q_8)\times_\phi \Z_3$ and $G/[G,G]=\Z_2$.
\item $G=H_{192}=2^4 D_{12}$: $[G,G]=2^4\Z_3$ and $G/[G,G]=(\Z_2)^2$.
\item $G=N_{72}=3^2 D_8$: $[G,G]=A_{3,3}$ and $G/[G,G]=(\Z_2)^2$.
\item $G=M_9=3^2 Q_8$:  $[G,G]=A_{3,3}$ and $G/[G,G]=(\Z_2)^2$.
\item $G=T_{48}=Q_8\times_\phi S_3$: $[G,G]=T_{24}=Q_8\times_\phi \Z_3$ and
$G/[G,G]=\Z_2$.
\end{itemize}

The crucial step in the proof of Theorem 1.0 is to determine the
possible fixed point set of an arbitrary element of $G$. This is done
by combining the analysis in our previous work \cite{CK1} with various
$G$-index theorems, and by exploiting various specific features of
the group $G$.

Here is the main technical input from \cite{CK1} (cf. Lemma 3.1 in \cite{CK1}). 
Since $b_2^{+}(X/G_0)=3\geq 2$, the canonical class $c_1(K_X)$ is represented by
$\sum_i n_i C_i$, where $n_i\geq 1$ and $\{C_i\}$ is a finite set of $J$-holomorphic 
curves such that (i) $\cup_i C_i$ is invariant under
the action of $G_0$, (ii) if $p\in X\setminus (\cup_i C_i)$
is fixed by an element $g\in G_0$, then the local representation of $g$
at $p$ must be contained in $SL_2(\C)$. (In particular, $p$ must be an
isolated fixed point of $g$, and all the $2$-dimensional components of the
fixed point set $\text{Fix}(g)$ are contained in $\cup_i C_i$.) 

The fixed point set of an element of order $2$ or $4$ is determined in the following

\begin{lemma}
(1) Let $g\in G$ be an involution. If $g\in G_0$, then $\text{Fix}(g)$
consists of $8$ isolated fixed points. If $g\in G\setminus G_0$, then
$\text{Fix}(g)$ is either empty or a disjoint union of embedded
$J$-holomorphic curves $\{\Sigma_j\}$ such that $c_1(K_X)\cdot \Sigma_j=0$
for each $j$.

(2) Let $g\in G_0$ be an element of order $4$. Then $\text{Fix}(g)$
consists of $4$ isolated fixed points, all with a local representation
contained in $SL_2(\C)$.
\end{lemma}

\begin{proof}
(1) Since $X$ is simply-connected, the action of $g$ can be lifted to
the spin structure, where there are two cases: $g$ is of even type,
meaning that the order of the lifting is $2$, and $g$ is of odd type,
meaning that the order of the lifting is $4$. Moreover,
$g$ has only isolated fixed points in the case of an even type, and $g$
is free or has only $2$-dimensional fixed components in the case of an odd
type (cf. \cite{AB}).

Suppose $g\in G_0$. Then $b_2^{+}(X/\langle g\rangle)=3$, and by \cite{Bry} 
$g$ is of even type with $8$ isolated fixed points. Now consider the case where
$g\in G\setminus G_0$. In this case $\text{Fix}(g)$ is either empty or is a disjoint 
union of embedded surfaces $\Sigma_j$ (cf. \cite{Bry}). Note that each 
$\Sigma_j$ is $J$-holomorphic 
because we choose $J$ to be $G$-equivariant.

We first show that $\sum_j c_1(K_X)\cdot \Sigma_j=0$. To see
this, suppose $t$ is the dimension of the $1$-eigenspace of $g$ in
$H^2(X;\R)$. Then by the Lefschetz fixed point theorem and
the $G$-signature theorem (cf. \cite{HZ}), we obtain
$$
\left\{\begin{array}{ccc}
2+ t-(22-t) & = & \sum_j\chi(\Sigma_j)\\
2(2-t) & = & -16 + \sum_j \frac{2^2-1}{3}\cdot\Sigma_j^2,\\
\end{array} \right .
$$
which gives $\sum_j(\chi(\Sigma_j)+\Sigma_j^2)=0$. By the adjunction
formula, we obtain
$$
\sum_j c_1(K_X)\cdot \Sigma_j=\sum_j-(\chi(\Sigma_j)+\Sigma_j^2)=0.
$$

On the other hand, recall from \cite{CK1} that $c_1(K_X)=\sum_i n_i C_i$
where $\{C_i\}$ is a finite set of $J$-holomorphic curves and
$n_i\geq 1$. For any $j$, if $\Sigma_j
\neq C_i$ for all $i$, then because of positivity of intersection of
$J$-holomorphic curves, $c_1(K_X)\cdot \Sigma_j\geq 0$ with equality iff
$\Sigma_j$ is disjoint from $\cup_i C_i$. If $\Sigma_j=C_i$ for some $i$,
then $c_1(K_X)\cdot \Sigma_j=c_1(K_X)\cdot C_i=0$ (cf. \cite{CK1}, Lemma 3.3).
In any event we have $c_1(K_X)\cdot \Sigma_j\geq 0$, which implies
$c_1(K_X)\cdot \Sigma_j=0$ because $\sum_j c_1(K_X)\cdot \Sigma_j=0$.

(2) Since $\text{Fix}(g)\subset\text{Fix}(g^2)$ and $g^2$ is an
involution in $G_0$, we see immediately that $g$ has only isolated fixed
points, with local representations of either type $(1,1)$, $(3,3)$, or
type $(1,3)$. We shall denote by $s_{+}$, $s_{-}$ the number of fixed
points of $g$ of type $(1,3)$ and type $(1,1)$ or $(3,3)$ respectively.
In order to determine $s_{+}$, $s_{-}$, we first compute with the Lefschetz
fixed point theorem and the $G$-signature theorem. To this end, it is useful
to observe that for the induced action of the involution $g^2$ on
$H^2(X;\R)$, the $1$-eigenspace has dimension $14$ and the $(-1)$-eigenspace
has dimension $8$. With this understood, if we denote by $t_{\pm}$ the
dimension of the $(\pm 1)$-eigenspace of $g$ in $H^2(X;\R)$, then
$t_{+}+t_{-}=14$. Now the Lefschetz fixed point theorem and the
$G$-signature theorem (cf. \cite{HZ}) give
rise to the following system of equations
$$
\left\{\begin{array}{ccc}
2+t_{+}-(14-t_{+}) & = & s_{+}+s_{-}\\
4(6-t_{+}) & = & -16+ 2s_{+}+(-2)s_{-},\\
\end{array} \right .
$$
where we use the assumption $g\in G_0$ so that $b_2^{+}(X/\langle g\rangle)=3$, and we use
the fact that the signature defect at a fixed point of type $(1,3)$ and
type $(1,1)$ or $(3,3)$ is $2,-2$ respectively, and the signature defect at a fixed point of $g^2$ is $0$. 
(This follows by a direct calculation using the formulas in \cite{HZ}.) The solutions
for $s_{+}$, $s_{-}$ (note that $s_{+}+s_{-}\leq 8$) are $s_{+}=4$ and
$s_{-}=0,2$ or $4$.

We proceed further by exploiting the fact that the action of $g$ can also
be lifted to the spin structure, and because $g^2$ is of even type, $g$ is
also of even type (i.e. a lifting of $g$ to the spin structure is of order
$4$). Moreover, the induced lifting of $g^2$ to the spin
structure is uniquely determined, i.e., it is independent of the different
choices of liftings of $g$ to the spin structure. With this understood,
the computation of the ``Spin-number'' $Spin(g^2,X)$ plays a crucial role
in the consideration which follows.

But first of all, a digression is needed in which we will recall a formula
for the local contribution of a fixed point to the ``Spin-number''
(cf. Lemma 3.8 of \cite{CK2}). Suppose $h$
is an order $p$ self-diffeomorphism ($p\geq 2$ and not necessarily prime)
which is spin and almost complex. Then because of the $h$-equivariant
almost complex structure, the $h$-equivariant spin structure corresponds
to an $h$-equivariant complex line bundle $L$, such that at an isolated
fixed point $m$ of local representation type $(a_m,b_m)$, the weight $r_m$
of the representation of $h$ on the fiber of $L$ at $m$ obeys
$2r_m+a_m+b_m=0 \pmod{p}$. Define $k(h,m)\equiv (2r_m+a_m+b_m)/p$.
Then the local contribution of $m$ to $Spin(h,X)$ is
$$
I_m=(-1)^{k(h,m)+1} \cdot\frac{1}{4}\csc(\frac{a_m\pi}{p})
\csc(\frac{b_m\pi}{p}).
$$
End of digression.

We apply the above formula to the involution $h\equiv g^2$. For each fixed
point $m$ of $g^2$, $(a_m,b_m)=(1,1)$, so that the local contribution
$I_m=\frac{1}{4}$ or $-\frac{1}{4}$, depending on whether $r_m=0$ or $1$. Now suppose
the $g^2$-index of the Dirac operator as a character is
$$
\text{index}_{g^2} \;D=d_0+d_1\C_1
$$
where $\C_k$ is the $1$-dimensional weight-$k$ representation. Then both
$d_0$ and $d_1$ are even integers because of the quaternion structure.
Since there are only $8$ fixed points and each contributes
$I_m=\frac{1}{4}$ or $-\frac{1}{4}$ to $Spin(g^2,X)$, it follows easily that
$Spin(g^2,X)=d_0-d_1$ only takes values of $-2$, $0$, or $2$.
One can further eliminate the possibility of $Spin(g^2,X)=0$ by
observing that $d_0+d_1=-\text{sign}(X)/8=2$ and that both $d_0,d_1$ are
even. The crucial consequence of the fact that $Spin(g^2,X)$ equals
either $-2$ or $2$ is that the weight $r_m$ of the representation of $g^2$
on the fiber of the complex line bundle $L$ is independent of the fixed
point $m$. This implies that for the element $g$, either $s_{+}=0$ or
$s_{-}=0$. Since $s_{+}=4$, $s_{-}$ must be
$0$, and the lemma follows.

\end{proof}

Next we discuss the fixed point set of an element of $G_0$ of an odd prime order. 
Unlike the cases we dealt with in Lemma 2.2, this requires analyzing the induced action on 
$\cup_i C_i$ in the way as we demonstrated in \cite{CK1}. In particular, we shall rely on several specific 
results from Section 3 of \cite{CK1}. We would like to point out that even though there is an additional 
assumption in \cite{CK1} that the action is trivial on $H^2(X;\R)$, this assumption is merely to ensure 
that each $(-2)$-sphere in $\cup_i C_i$ is invariant under the action. 

We recall that the connected components of $\cup_i C_i$ may be divided into the following three 
types (cf. Section 3 in \cite{CK1}):
\begin{itemize}
\item [{(A)}] A single $J$-holomorphic curve of self-intersection $0$
which is either an embedded torus, or a cusp sphere, or a nodal sphere.
\item [{(B)}] A union of two embedded $(-2)$-spheres intersecting at a
single point with tangency of order $2$.
\item [{(C)}] A union of embedded $(-2)$-spheres intersecting transversely.
\end{itemize}
Furthermore, a type (C) component may be conveniently represented by one of the graphs
of type $\tilde{A}_n$, $\tilde{D}_n$, $\tilde{E}_6$, $\tilde{E}_7$
or $\tilde{E}_8$ listed in Figure 1, where a vertex in a graph represents
a $(-2)$-sphere and an edge connecting two vertices represents a transverse,
positive intersection point of the two $(-2)$-spheres represented by the
vertices.

\begin{figure}[ht]

\setlength{\unitlength}{1pt}

\begin{picture}(420,420)(-210,-195)

\linethickness{0.5pt}

\put(-170,165){$\tilde{A}_n$}

\put(-120,165){\circle*{3}}

\put(-120,165){\line(1,1){30}}

\put(-90,195){\circle*{3}}

\put(-90,195){\line(1,0){30}}

\put(-60,195){\circle*{3}}

\put(-60,195){\line(1,-1){30}}

\put(-30,165){\circle*{3}}

\put(-120,165){\line(1,-1){30}}

\put(-90,135){\circle*{3}}

\put(-90,135){\line(1,0){30}}

\put(-60,135){\circle*{3}}

\multiput(-60,135)(3,3){10}{\line(1,1){2}}

\put(20,165){($n+1$ vertices, $n\geq 1$)}

\put(-170,90){$\tilde{D}_n$}

\put(-120,105){\circle*{3}}

\put(-120,105){\line(2,-1){30}}

\put(-120,75){\circle*{3}}

\put(-120,75){\line(2,1){30}}

\put(-90,90){\circle*{3}}

\put(-90,90){\line(1,0){30}}

\put(-60,90){\circle*{3}}

\put(-30,90){\circle*{3}}

\multiput(-60,90)(3,0){10}{\line(1,0){2}}

\put(0,90){\circle*{3}}

\put(-30,90){\line(1,0){30}}

\put(0,90){\line(2,1){30}}

\put(0,90){\line(2,-1){30}}

\put(30,105){\circle*{3}}

\put(30,75){\circle*{3}}

\put(60,90){($n+1$ vertices, $n\geq 4$)}

\put(-170,30){$\tilde{E}_6$}

\put(-120,30){\circle*{3}}

\put(-120,30){\line(1,0){30}}

\put(-90,30){\circle*{3}}

\put(-90,30){\line(1,0){30}}

\put(-60,30){\circle*{3}}

\put(-60,30){\line(1,0){30}}

\put(-30,30){\circle*{3}}

\put(-30,30){\line(1,0){30}}

\put(0,30){\circle*{3}}

\put(-60,30){\line(0,-1){60}}

\put(-60,0){\circle*{3}}

\put(-60,-30){\circle*{3}}

\put(-170,-60){$\tilde{E}_7$}

\put(-120,-60){\circle*{3}}

\put(-120,-60){\line(1,0){180}}

\put(-90,-60){\circle*{3}}

\put(-60,-60){\circle*{3}}

\put(-30,-60){\circle*{3}}

\put(0,-60){\circle*{3}}

\put(30,-60){\circle*{3}}

\put(60,-60){\circle*{3}}

\put(-30,-60){\line(0,-1){30}}

\put(-30,-90){\circle*{3}}

\put(-170,-150){$\tilde{E}_8$}

\put(-120,-150){\circle*{3}}

\put(-120,-150){\line(1,0){210}}

\put(-90,-150){\circle*{3}}

\put(-60,-150){\circle*{3}}

\put(-30,-150){\circle*{3}}

\put(0,-150){\circle*{3}}

\put(30,-150){\circle*{3}}

\put(60,-150){\circle*{3}}

\put(90,-150){\circle*{3}}

\put(-60,-150){\line(0,-1){30}}

\put(-60,-180){\circle*{3}}

\end{picture}

\caption{}

\end{figure}

With the preceding understood, let $g\in G_0$ be an element of order $3$. 
Then $\text{Fix}(g)$ may be
divided into subsets (or groups) of the following four types.
\begin{itemize}
\item [{(I)}] One fixed point with local representation in $SL_2(\C)$.
\item [{(II)}] Three fixed points, all with local representation of type
$(k,k)$ for some $k\neq 0 \mod{3}$.
\item [{(III)}] One fixed point of local representation type $(k,k)$,
$k\neq 0 \mod{3}$, and one fixed spherical component of self-intersection $-2$.
\item [{(IV)}] One fixed toroidal component of self-intersection $0$.
\end{itemize}
Moreover, a group of fixed points of type (III) comes only
from a type (C) component of $\cup_i C_i$. For the sake of later arguments
in this section, we shall give below a brief analysis of the action of $g$
on a type (C) component of $\cup_i C_i$.  

Let $\Lambda$ be a type (C) component which is invariant under $g$.  Then there is an
induced action of $g$ on the graph representing $\Lambda$. We consider first the case where
$\Lambda$ is represented by a $\tilde{A}_n$ graph. Then the induced action
of $g$ on the graph is either a trivial action or a rotation. If the induced action is trivial, 
then the fixed points of $g$ contained in $\Lambda$ are either entirely
of type (I) or consist of $(n+1)/3$ groups of type (III) fixed points (cf. Proposition 3.7 in \cite{CK1}). 
(We note that by Lemma 3.6 in \cite{CK1}, $\Lambda$ can not be a union of three $(-2)$-spheres intersecting transversely at a single point in this case.)
If the induced action is a rotation, then either $\Lambda$ contains no fixed points of $g$, 
or $\Lambda$ is a union of three $(-2)$-spheres intersecting transversely at a single point, in which case
the intersection point is the only fixed point of $g$ contained in $\Lambda$ and it is a type (I) fixed point. 

Next we assume that $\Lambda$ is represented by a $\tilde{E}_6$ graph. We claim that
the induced action on the graph must be non-trivial. To see this, suppose the induced action
of $g$ on the graph is trivial. If we denote by $C_0$ the $(-2)$-sphere in $\Lambda$ which is
represented by the vertex of the graph that are adjacent to three other vertices (i.e., the central
vertex), then $C_0$ must be fixed under the action of $g$ (because a nontrivial cyclic group action 
on $\s^2$ has exactly $2$ fixed points). Let $C_1$ be a $(-2)$-sphere in $\Lambda$ intersecting with 
$C_0$ and $C_2$ be the $(-2)$-sphere intersecting with $C_1$. Then by Lemma 3.6 of \cite{CK1}, the rotation numbers at the two fixed points associated to $C_1$ must be $(0,1)$ and $(1,1)$, with $(0,1)$ 
being the rotation numbers at the intersection point of $C_0$ and $C_1$. (See Section 3 in \cite{CK1} 
for a discussion on rotation numbers.) This implies, by Lemma 3.6 of \cite{CK1} again, that the rotation
numbers at the two fixed points of $g$ associated to $C_2$ are $(1,1)$ and $(0,1)$, with $(1,1)$ being
the rotation numbers at the intersection point of $C_1$ and $C_2$. It follows, since the rotation numbers
at the other fixed point on $C_2$ are $(0,1)$, that $C_2$ must intersect with a $2$-dimensional 
component of the fixed point set of $g$. But this is clearly a contradiction, hence our claim that the
induced action of $g$ on the graph must be non-trivial. With this understood, it is easily seen that 
$\Lambda$ contains exactly two fixed points of $g$, and these two fixed points are on the 
$(-2)$-sphere $C_0$.  Furthermore, it follows from Lemma 3.6 in \cite{CK1} that these two
fixed points are of type (I).  By a similar argument, we show that $\Lambda$ can not be represented 
by a $\tilde{E}_8$ graph, because a $\tilde{E}_8$ graph admits no non-trivial actions of $g$. 

Suppose $\Lambda$ is represented by a $\tilde{D}_n$ graph. Then the induced action on the graph
must be trivial, and the fixed points of $g$ contained in $\Lambda$ consist
of $1$ group of type (II) fixed points and $(n-1)/3$ groups of type (III)
fixed points. Suppose $\Lambda$ is represented by a $\tilde{E}_7$ graph,
then the induced action on the graph must be trivial and $\Lambda$ gives
rise to $3$ groups of type (III) fixed points of $g$. (See Proposition 3.7 in \cite{CK1}.)

Finally, it is helpful to note that only a type $\tilde{A}_n$, $\tilde{D}_n$, or $\tilde{E}_7$ component of 
$\cup_i C_i$ can possibly contain a group of type (III) fixed points of $g$, and only a type $\tilde{D}_n$ component of $\cup_i C_i$ can contain a group of type (II) fixed points of $g$. 

The fixed point set of an element of order $3$ in $G_0$ is described in the following

\begin{lemma}
Suppose $g\in G_0$ is an element of order $3$. Let $u,v$ and $w$ be the
number of groups of type (I), (II) and (III) fixed points of $g$
respectively, and let $t=b_2(X/\langle g\rangle)$. Then
\begin{itemize}
\item [{(1)}] $2u+3v=12$, $w\leq 6$ and $t\geq 10$.
Moreover, $t=10$ iff $(u,v,w)=(6,0,0)$.
\item [{(2)}] Suppose $w=0$. If there exist $3$ distinct involutions
$h_1,h_2,h_3\in G_0$ each of which commutes with $g$, then $(u,v)=(6,0)$.
\item [{(3)}] Suppose $w=0$. If $g$ is contained in a subgroup of $G_0$
which is isomorphic to $T_{24}$, then $(u,v)=(6,0)$.
\end{itemize}
\end{lemma}

\begin{proof}
(1) Note that a toroidal fixed component $Y$ of $g$ does not make any
contribution
in the Lefschetz fixed point theorem because $\chi(Y)=0$, nor does it
contribute in the $G$-signature theorem because $Y\cdot Y=0$. Hence we shall
ignore it in our calculations below.

Observe that $t=b_2(X/\langle g\rangle)$ is the dimension of the $1$-eigenspace of $g$
in $H^2(X;\R)$, and that $t-(22-t)/2$ is the trace of $g$ on $H^2(X;\R)$.
Hence the Lefschetz fixed point theorem and the $G$-signature theorem
give rise to the following equations
$$
\left\{\begin{array}{ccc}
2+ t-(22-t)/2 & = & u+3v+3w\\
3(6-t) & = & -16 + \frac{2}{3}u-2v-6w\\
\end{array} \right .
$$
where we make use of $b_2^{+}(X/\langle g\rangle)=3$ and the fact that the total signature
defect for a group of type (I), (II) and (III) fixed points is $\frac{2}{3}$, $-2$ and $-6$ respectively.
(The claim concerning the total signature defect follows by a direct calculation 
using the formulas in \cite{HZ}.)
The equation $2u+3v=12$ follows immediately, which has $3$ solutions:
$(u,v)=(6,0), (3,2)$, and $(0,4)$. The inequality $w\leq 6$ follows from
$u+3v\geq 6$ and the fact that $t\leq b_2(X)=22$. It is also easy
to check that $t\geq 10$, with $t=10$ iff $(u,v,w)=(6,0,0)$.

(2) Suppose $(u,v)=(0,4)$, in which case
$g$ has $12$ isolated fixed points. From the analysis of a possible action of $g$ on a type (C)
component preceding Lemma 2.3, we see that only a component represented by a type
$\tilde{D}_n$ graph can possibly contain a group of type (II) fixed points, and at the same time,
there must be groups of type (III) fixed points. Since we assume that $w=0$, these $12$ points 
can not be contained in type (C) components of $\cup_i C_i$. It follows by Proposition 3.7 of \cite{CK1}
that the $12$ fixed points of $g$ must be contained in $4$ toroidal components of $\cup_i C_i$, where
each toroidal component contains exactly $3$ isolated fixed points. 

By Lemma 2.2 each $h_i$ has $8$ isolated fixed points.  Since $g$ and $h_i$ commute, there 
is an induced action of $g$ on $\text{Fix}(h_i)$, and it follows that $g$ and $h_i$ must have at least 
$2$ common fixed points. This implies that one of the toroidal components containing the fixed points 
of $g$ is invariant under $h_i$, and consequently,
$g$ and $h_i$ generate an effective cyclic action of order $6$ on that torus. Since an order-$6$
cyclic action on a torus is either free or has only $1$ fixed point, we see that 
distinct common fixed points of $g$ and $h_i$ are contained in distinct 
toroidal components of $\cup_i C_i$. It follows easily that there
are $i,j$ with $i\neq j$ such that $h_i$ and $h_j$ leave one of the toroidal
components invariant, because for each $i$, $g$ and $h_i$ have at least
$2$ common fixed points and there are exactly $4$ toroidal components of
$\cup_i C_i$ containing the fixed points of $g$. But this is easily seen
a contradiction, as $h_i$ acts freely on the set of common fixed points
of $g$ and $h_j$ because $\text{Fix}(h_i)\cap \text{Fix}(h_j)=\emptyset$.
The case where $(u,v)=(3,2)$ can be similarly eliminated.
This proves that $(u,v)=(6,0)$.

(3) Note that $T_{24}=Q_8\times_\phi\Z_3$, where we may assume without
loss of generality that the action of $\Z_3=\langle g\rangle$ on
$$
Q_8=\{i,j,k|i^2=j^2=k^2=-1,\; ij=k,\; jk=i,\; ki=j\}
$$
is given by $\phi(g)(i)=j$, $\phi(g)(j)=k$ and $\phi(g)(k)=i$.

By Lemma 2.2, it follows easily that $Q_8$ has either $2$ or $4$ isolated
fixed points (see e.g. \cite{Xiao}). Since there is an
induced action of $g$ on the fixed point set of $Q_8$, we see immediately
that $T_{24}$ has at least $1$ fixed point.

Suppose $(u,v)=(0,4)$. As we argued in (2) above, at least
one of the $4$ toroidal components must be invariant under $T_{24}$
because it contains a fixed point of $T_{24}$. But this is impossible
as there are no such $T_{24}$-actions on the torus (cf. \cite{Sct}).

If $(u,v)=(3,2)$, then $g$ and $-1\in Q_8$ must have $5$ common fixed points.
It follows as we argued in (2) above that each of the $2$ toroidal components
of $\cup_i C_i$ which contains the type (II) fixed points of $g$ must be
invariant under $-1$, with each containing exactly $1$ common fixed point of
$g$ and $-1$. But on the other hand, by Proposition 3.7 in \cite{CK1},
each of the $2$ toroidal components contains exactly $4$ fixed points of
$-1$, so that all of the fixed points of $-1$ are contained in there.
This is a contradiction to the fact that the $3$ type (I) fixed points of
$g$, which are not contained in the $2$ toroidal components, are also fixed
under $-1$. Hence the case where $(u,v)=(3,2)$ is also ruled out.

\end{proof}

\noindent{\bf Proof of Theorem 1.0}:

\vspace{3mm}

The general strategy goes as follows. For each of the $6$ maximal symplectic $K3$ groups
listed in Theorem 1.0, there is a subgroup $H\subset G_0$ such that for any symplectic holomorphic 
action of $H$ on a $K3$ surface, one has $\mu(H)=5$ where
$$
\mu(H)\equiv \frac{1}{|H|}\sum_{g\in H} tr(g). 
$$
(See \cite{Mu, Xiao} for the calculation of $\mu(H)$ for a symplectic automorphism
group $H$ of a $K3$ surface.) The main task in the proof of Theorem 1.0 is to show
that any effective action of $H$ on $X$ via symplectic symmetries must have the same 
fixed point set as does a symplectic holomorphic action of $H$ (except for possible toroidal 
fixed components).  As a consequence this implies that 
$$
\dim (H^2(X;\R))^H=\dim (H^\ast(X;\R))^H-2=\mu(H)-2=5-2=3. 
$$
On the other hand, $b_2^{+}(X/H)=3$ because $H\subset G_0$, so that $H^2(X;\R)^H$ 
must be positive-definite. It follows that $c_1(K_X)=0$ because $c_1(K_X)\in H^2(X;\R)^H$
and $c_1(K_X)\cdot c_1(K_X)=0$, which is equivalent to $X$ being minimally exotic, i.e., $r_X=0$.

\vspace{3mm}

{\it Case (1)}. $G=L_2(7)$. First note that $G_0=G$, i.e., $b_2^{+}(X/G)=3$.  An element of $G$ is
of order $2$, $3$, $4$, or $7$. The following lemma describes the fixed point set of an order-$7$
element of $G$. 

\begin{lemma}
Let $g\in G=L_2(7)$ be any element of order $7$. Then $g$ has exactly
$3$ isolated fixed points, and is either pseudofree or has only 
toroidal fixed components.
\end{lemma}

\begin{proof}
We first show that if a type (C) component of $\cup_i C_i$ 
contains a fixed point of $g$, then its local representation at the fixed point must lie in 
$SL_2(\C)$. To this end, we recall that the normalizer of $\langle g\rangle$ in $G$
is a maximal subgroup $D$ of order $21$ which is a semi-direct product of
$\Z_7$ by $\Z_3$ (cf. \cite{Con}). Let $\Lambda$ be a type (C)
component which contains a fixed point of $g$. (Note that $\Lambda$ is invariant under $g$.)
If it is represented by a type $\tilde{D}_n$, $\tilde{E}_6$,
$\tilde{E}_7$ or $\tilde{E}_8$ graph, then since the $(-2)$-spheres in
$\Lambda$ generate a lattice in $H_2(X;\Z)$ which contains a
negative-definite sublattice of rank at least $4$, the orbit of $\Lambda$
under the action of $G$ can have at most $4$ components because of the
constraint $b_2^{-}(X)=19$. On the other hand, one can easily check that
$\Lambda$ is not invariant under $G=L_2(7)$, and since the index of the
maximal subgroup $D$ is $8$, there are at least $8$ components in the orbit
of $\Lambda$, which is a contradiction. Suppose $\Lambda$ is represented by a type
$\tilde{A}_n$ graph. Then each $(-2)$-sphere in $\Lambda$ is invariant under $g$ (since
$\Lambda$ contains a fixed point of $g$). By Proposition 3.7 in \cite{CK1}, there are $3$
possibilities: (i) $n=-1\mod{7}$, so that $\Lambda$ contains at
least seven $(-2)$-spheres, (ii) $\Lambda$ only contains fixed points of $g$ whose 
local representations lie in $SL_2(\C)$, (iii) $\Lambda$ is a union of three $(-2)$-spheres 
intersecting transversely at one single point. Note that case (i) can be eliminated by a similar
argument as in the cases of type $\tilde{D}_n$, $\tilde{E}_6$, $\tilde{E}_7$ or $\tilde{E}_8$ graphs.
Case (iii) is ruled out as follows.  Note that the maximal subgroup $D$ can not act
linearly and freely on $\s^3$, so that such a $\Lambda$ can not be invariant
under the action of $D$. Hence if such a $\Lambda$ exists, there
must be at least $3\times 8=24$ components in the orbit of $\Lambda$ under
the action of $G$. But this is impossible because of the constraint
$b_2^{-}(X)=19$.  This finishes the proof of our claim.  

Secondly, we will show that there are no type (B) components which contain
a fixed point of $g$. Suppose $\Lambda$ is such a type (B) component. One can check 
easily that $\Lambda$ can not be invariant under the action of $D$, so that there are at
least $24$ type (B) components in $\cup_i C_i$. But this contradicts the fact
that $b_2^{-}(X)=19$. Hence there are no type (B) components containing
a fixed point of $g$. 

Finally, suppose a type (A) component $\Lambda$ of $\cup_i C_i$ contains a fixed point of $g$.
Then by Proposition 3.7 in \cite{CK1}, $\Lambda$ is either a fixed toroidal component, or $\Lambda$
is a cusp sphere containing $2$ fixed points of $g$ of local representations of type $(2k,3k)$,
$(-k,6k)$ for some $k\neq 0 \mod{7}$ respectively.

With the preceding understood, we conclude that $g$ has only fixed toroidal components, 
and that the isolated fixed points of $g$ can be divided into groups of the following two types:
\begin{itemize}
\item [{(1)}] One fixed point with local representation in $SL_2(\C)$.
\item [{(2)}] Two fixed points with local representation of type $(2k,3k)$,
              $(-k,6k)$ for some $k\neq 0 \mod{7}$ respectively.
\end{itemize}

Next we compute with the Lefschetz fixed point theorem and the $G$-signature theorem.
Denote by $t$ the dimension of the $1$-eigenspace of $g$ in $H^2(X;\R)$
(note that $22-t$ must be divisible by $6$), and denote by $u,v$ the
number of groups of type (1), (2) isolated fixed points of $g$ respectively.
Then by the Lefschetz fixed point theorem and the $G$-signature theorem,
$$
\left\{\begin{array}{ccc}
2+ t-(22-t)/6 & = & u+2v\\
7(6-t) & = & -16 +10 u-8v,\\
\end{array} \right .
$$
where we make use of $b_2^{+}(X/\langle g\rangle)=3$ and the fact that the total signature
defect for a group of type (1), (2)  fixed points of $g$ is $10$ and $-8$
respectively (cf. \cite{CK1}, Lemma 3.8). The solutions to the above system of equations are
$$
(t,u,v)=(4,3,0), (10,2,4), (16,1,8), (22,0,12).
$$

The cases where $(t,u,v)=(10,2,4)$ or $(16,1,8)$ can be ruled out as follows.
The maximal subgroup $D$ induces a $\Z_3$-action on the set of isolated
fixed points of $g$, which must be free because $D$ can not act freely and
linearly on $\s^3$. This implies that the number of fixed points, which is
$u+2v$, must be divisible by $3$.

In the case of $(t,u,v)=(22,0,12)$, $g$ is homologically trivial. 
Since $G=L_2(7)$ is a simple group, it follows that the action of $G$
is also homologically trivial.  But, because $G$ is nonabelian, 
this is impossible by McCooey's theorem in \cite{McC}.

The only case left is $(t,u,v)=(4,3,0)$, which shows that $g$ has exactly
$3$ isolated fixed points.

\end{proof}

Next we consider the action of an element $g\in G$ of order $3$. We claim
that $g$ has exactly $6$ isolated fixed points, with possibly some fixed
toroidal components. To see this, we note that there is an element $h\in G$
of order $7$ such that $D=\langle g,h\rangle$ is a nonabelian subgroup of order $21$,
which is the normalizer of $\langle h\rangle$ (cf. \cite{Con}). From the proof of Lemma 2.4,
we see that the dimension of the $\exp(\frac{2\pi ik}{7})$-eigenspace
of $h$ in $H^2(X;\R)$ is $\frac{22-4}{6}=3$ for each $1\leq k\leq 6$.
By examining the action of $D$ on the $\exp(\frac{2\pi ik}{7})$-eigenspaces
of $h$, $1\leq k\leq 6$,
one can check easily that the dimension of the $1$-eigenspace of $g$
in $H^2(X;\R)$ is at most $10$. By Lemma 2.3 (1), our claim follows.

Now with Lemma 2.2, which describes the number of fixed points of an
element of order $2$ or $4$, we see that for any $g\in G$, the Lefschetz
fixed point theorem implies that the trace $tr(g)$ is the same as that of a 
symplectic automorphism of order $|g|$ on a $K3$ surface. By Mukai \cite{Mu}, 
$\mu(G)=5$ for a symplectic holomorphic $G=L_2(7)$ action. This implies that
$$
\dim (H^\ast(X;\R))^G=\mu(G)\equiv \frac{1}{|G|}\sum_{g\in G} tr(g)=5.
$$
As we pointed out in the beginning of the proof of Theorem 1.0, this implies that 
$c_1(K_X)=0$, and hence $X$ is minimally exotic. 

\vspace{2mm}

End of Case (1).

\vspace{2mm}

{\it Case (2)}. $G=M_{20}$ or $A_6$. In this case we shall exploit the fact that there is
a subgroup of $G_0$ which is isomorphic to either $A_5$ or $A_6$. 

\begin{lemma}
Suppose $H\subset G_0$ is a subgroup isomorphic to either $A_5$
or $A_6$. Let $g\in H$ be an element of odd order. Then $g$ is either
pseudofree or has only toroidal fixed components. Moreover, $g$ has 
$4$ isolated fixed points if $|g|=5$, and $g$ has either $6$ or $12$ isolated 
fixed points when $|g|=3$.
\end{lemma}

\begin{proof}
Suppose $g\in H$ is an element of order $5$. Without loss of generality
we may assume that $H\cong A_5$, because in the case of $H\cong A_6$,
$g$ is contained in an $A_5$-subgroup of $H$. With this understood,
the maximal subgroup of $H$ containing $g$ is a dihedral group
$D_{10}\subset H$ of index $6$ (cf. \cite{Con}). One can similarly argue,
as in the proof of Lemma 2.4, that if a type (C) component of
$\cup_i C_i$ contains a fixed point of $g$, then it must be represented by
a type $\tilde{A}_n$ graph and the fixed point is of local representation lying in 
$SL_2(\C)$.

By Proposition 3.7 in \cite{CK1}, if a type (A) component $\Lambda$ of $\cup_i C_i$
contains a fixed point of $g$, then $\Lambda$ is either a fixed toroidal component, or
$\Lambda$ is a cusp or nodal sphere containing only fixed points of $g$ of local 
representation lying in $SL_2(\C)$. If a type (B) component $\Lambda$ 
contains a fixed point of $g$, then $\Lambda$ contains three fixed points of $g$, 
one with local representation of type $(k,2k)$ and the other two of type $(-k,4k)$ for 
some $k\neq 0 \mod{5}$. 

In conclusion, $g$ has only toroidal fixed components and the isolated fixed points of 
$g$ can be divided into groups of the following two types:
\begin{itemize}
\item [{(1)}] One fixed point with local representation in $SL_2(\C)$.
\item [{(2)}] Three fixed points, one with local representation of type
$(k,2k)$ and the other two of type $(-k,4k)$ for some $k\neq 0 \mod{5}$.
\end{itemize}

Denote by $t$ the dimension of the $1$-eigenspace of $g$ in $H^2(X;\R)$
(note that $22-t$ must be divisible by $4$), and denote by $u,v$ the number
of groups of type (1), (2) isolated fixed points of $g$ respectively. Then
by the Lefschetz fixed point theorem and the $G$-signature theorem,
$$
\left\{\begin{array}{ccc}
2+ t-(22-t)/4 & = & u+3v\\
5(6-t) & = & -16 + 4u-8v,\\
\end{array} \right .
$$
where we make use of $b_2^{+}(X/\langle g\rangle)=3$ and the fact that the total signature
defect for a group of type (1), (2)  fixed points is $4$ and $-8$
respectively (cf. \cite{CK1}, Lemma 3.8). The solutions to the above
system of equations are
$$
(t,u,v)=(6,4,0), (10,3,2), (14,2,4), (18,1,6), (22,0,8).
$$

The cases where $u=1$ or $3$ can be eliminated as follows. There is an
involution on the set of isolated fixed points of $g$ induced by the action
of $D_{10}$, which is free because $D_{10}$ can not act freely and linearly
on $\s^3$. Consequently, the number of isolated fixed points of $g$ must be
divisible by $2$. To eliminate the case where $(t,u,v)=(14,2,4)$, note that
in this case $\cup_i C_i$ has $4$ type (B) components each of which contains
a fixed point of $g$. Moreover, it is easy to see that each component is not
invariant under the action of $D_{10}$. Since the index of $D_{10}\subset H$ is $6$,
there are at least $4\times 2 \times 6=48$ type (B) components of $\cup_i C_i$, which 
contradicts $b_2^{-}(X)=19$. Finally, the case where $(t,u,v)=(22,0,8)$ is ruled out by 
McCooey's theorem \cite{McC} because $H$ is simple and nonabelian. Hence $g$ 
has $4$ isolated fixed points when $|g|=5$.

Next suppose $g\in H$ is an element of order $3$, where $H$ is either
$A_5$ or $A_6$. We claim that $\text{Fix}(g)$ does not contain any group
of type (III) fixed points (i.e., $w=0$ in Lemma 2.3). To see this, note first
that there are no type (C) components of $\cup_i C_i$ which are represented 
by a graph of type $\tilde{D}_n$, $\tilde{E}_6$, $\tilde{E}_7$ or $\tilde{E}_8$.  
The point is that such a component can not contain any fixed points of an order-$5$
element of $H$, hence can not be invariant under the action of an order-$5$ element. 
If such a component exists, then 
there are at least $5$ such components in $\cup_i C_i$, and this contradicts 
$b_2^{-}(X)=19$. Hence if $\text{Fix}(g)$ contains a group of type (III) fixed points, 
it must come from a type (C) component $\Lambda$ which is represented by a
type $\tilde{A}_n$ graph, where $n=-1\mod{3}$ (cf. Proposition 3.7 in \cite{CK1}).
Let $h\in H$ be an order-$5$ element. Since $g\neq hgh^{-1}$, it follows easily 
that $h$ and $g$ can not have a common isolated fixed point in $\Lambda$, which 
implies that either $\Lambda$ is not invariant under $h$ or $h$ acts freely on $\Lambda$.
In any event, the case where $n>2$ can be ruled out by using the fact
$b_2^{-}(X)=19$. To eliminate the case where $n=2$, we note that there is
a subgroup $K\subset H$ which is isomorphic to the symmetric group $S_3$
and contains $\langle g\rangle$ as a normal subgroup. Clearly $K$ can not leave
$\Lambda$ invariant if it is represented by a $\tilde{A}_2$ graph, so
that $\Lambda$ must come in pairs. Again this is impossible by the fact
that $b_2^{-}(X)=19$. Hence $\text{Fix}(g)$ does not contain any group
of type (III) fixed points. The action of $K$ on $\text{Fix}(g)$ also
implies that $u$ is even in Lemma 2.3 (because $S_3$ can not act freely
and linearly on $\s^3$). Hence $g$ has either $6$ or $12$ isolated fixed
points when $|g|=3$.

Finally, note that $g$ is either pseudofree or has only toroidal fixed components.

\end{proof}

Let $G=M_{20}$. Since $[G,G]=G$, we see that $G_0=G$. We claim that
for each $g\in G$ the trace $tr(g)$ on $H^\ast(X;\R)$ is the same as
that of a symplectic automorphism of order $|g|$ on a $K3$ surface.
With this the proof of Theorem 1.0 proceeds identically as in the
case of $L_2(7)$, as for $G=M_{20}$,
$\mu(G)\equiv |G|^{-1}\sum_{g\in G}tr(g)=5$
is also true for a symplectic holomorphic action (cf. \cite{Mu}). When $|g|\neq 3$
or $6$, our claim follows readily from Lemma 2.2 and Lemma 2.5. For
the case where $|g|=3$ or $6$, we need to argue with some extra information
about the structure of $G=M_{20}$.

According to Mukai \cite{Mu}, page 189, $M_{20}=2^4 A_5$, where the action
of $A_5$ on $2^4$ is obtained by realizing $2^4$ as the hypersurface
$V=\{(a_i)|\sum_{i=1}^5 a_i=0\}\subset (\Z_2)^5$ with $A_5$ acting
as permutations
of the $5$ coordinates. Clearly, for each element $g$ of order $3$ in $A_5$,
there are $3$ nonzero elements of $V$ which are fixed under $g$. This gives
$3$ distinct involutions in $G$, each of which commutes with $g$.
By Lemma 2.3 (2), $g$ has $6$ isolated fixed points. It also follows easily
from the proof of Lemma 2.3 (2) that an order $6$ element of $G$ has $2$
isolated fixed points, with possibly some fixed toroidal components. In
conclusion, for an order $3$ or $6$ element $g\in G$, the trace $tr(g)$ on
$H^\ast(X;\R)$ is also the same as that of a symplectic automorphism on a
$K3$ surface of the same order. This completes the proof for the case where
$G=M_{20}$.

Let $G=A_6$. In this case, we also have $G_0=G$. As in the case of $M_{20}$, 
it suffices to show that for each $g\in G$ with $|g|=3$, there are $6$ isolated fixed
points. (Note that $\mu(G)=5$ is also true for a symplectic holomorphic $A_6$-action 
(cf. \cite{Mu})). To this end, we recall the following fact about $A_6$: There are $2$ 
conjugacy classes of elements of order $3$ in $A_6$; the centralizer of each order $3$ 
element in $A_6$ is isomorphic to $(\Z_3)^2$, hence has order $9$. Now suppose an 
element $g$ of order $3$ in $G=A_6$ has, instead, $12$ isolated fixed points.
Then the conjugacy class of $g$ will make an increase of $\frac{6}{9}=\frac{2}{3}$ to
$$
\mu(G)\equiv \frac{1}{|G|}\sum_{g\in G}tr(g)
$$
when compared with a holomorphic $A_6$-action. Since there are only two
conjugacy classes of elements of order $3$ in $A_6$, a nonzero increase
to $\mu(G)$ is either $\frac{2}{3}$ or $\frac{4}{3}$, neither of which is
integral. This shows that an element of order $3$ in $G$ must have $6$
isolated fixed points, and the proof of Theorem 1.0 for the case
of $G=A_6$ follows.

\vspace{2mm}

End of Case (2) where $G=M_{20}$ or $A_6$.

\vspace{2mm}

{\it Case (3)}. $G=A_{4,4}$.  Let $H\equiv [G,G]=A_4\times A_4$.
Then since $[G,G]\subset G_0$, we have $b_2^{+}(X/H)=3$. Note that
$\mu(H)=5$ for a symplectic automorphism group $H$ of a $K3$ surface (cf. \cite{Xiao}).
Hence by Lemma 2.2, it suffices to show that for each $g\in H$ of order
$3$, the trace $tr(g)$ on $H^\ast(X;\R)$ is the same as that of a
symplectic automorphism of order $3$ on a $K3$ surface.

There are $4$ conjugacy classes of order $3$ elements
in $G$, which are represented by
$(g,1),(1,g),(g,g),(g,g^2)\in A_4\times A_4=H$ for some
fixed element $g\in A_4$ of order $3$. Since the trace on $H^\ast(X;\R)$
only depends on the conjugacy class in $G$, it suffices to examine these
$4$ elements of $H$.

We first show that there are no type (III) fixed points (i.e., $w=0$ in
Lemma 2.3). Consider the case $(g,1)$ first. The normalizer of $\langle (g,1)\rangle$
in $H$ is $\langle g\rangle\times A_4$ which has index $4$. If $\Lambda$ is a type (C)
component of $\cup_i C_i$ which contains a group of type (III) fixed points
of $(g,1)$, then the fact $b_2^{-}(X)=19$ immediately rules out the
possibility that $\Lambda$ is represented by a $\tilde{E}_7$ graph or
a $\tilde{A}_n$ graph where $n\neq 2$. If $\Lambda$ is represented by a
$\tilde{D}_n$ graph or a $\tilde{A}_2$ graph, then one can check easily
that the orbit of $\Lambda$ under the normalizer $\langle g\rangle\times A_4$ has
at least $3$ components. This also
contradicts $b_2^{-}(X)=19$, and hence there are no type (III) fixed points
of $(g,1)$. The case of $(1,g)$ is completely parallel.
For the case of $(g,g)$ or $(g,g^2)$, the normalizer of $\langle (g,g)\rangle$ or
$\langle (g,g^2)\rangle$ in $H$ is $\langle g\rangle\times \langle g\rangle$ which has index 
$16$. It follows immediately from $b_2^{-}(X)=19$ that there are no type (III) fixed points.

Now by Lemma 2.3 (2), each of $(g,1)$ and $(1,g)$ has exactly $6$ isolated
fixed points. The case of $(g,g)$ or $(g,g^2)$ is more involved, which is
addressed in the following

\begin{lemma}
Suppose $c_1(K_X)\neq 0$. Then the number of isolated fixed points of
$(g,g)$ or $(g,g^2)$ is even.
\end{lemma}

\begin{proof}
We consider the case of $(g,g)$ only. The case of $(g,g^2)$ is completely
parallel.

By Lemma 2.3, the number of isolated fixed points of $(g,g)$ is either
$6$, $9$ or $12$. Suppose to the contrary that it is $9$. A contradiction
is derived as follows. Observe that there is an involution
$h\in G\setminus H$ such that $h$ and $(g,g)$ generate a subgroup $K$ of $G$,
where $K$ is isomorphic to $S_3$ and $\langle (g,g)\rangle$ is a normal subgroup of $K$.
There is an induced action of $K$ on $\text{Fix}((g,g))$, which preserves
the type of the fixed points. Since $(g,g)$ has $3$ type (I) fixed points,
one of them, denoted by $p$, must be fixed by $K$. Note that $K\cong S_3$
can not have an isolated fixed point, hence $h\in G\setminus G_0$ and
$\text{Fix}(h)$ consists of a disjoint union of embedded $J$-holomorphic
curves $\{\Sigma_j\}$ where $c_1(K_X)\cdot\Sigma_j=0$ for each $j$
(cf. Lemma 2.2 (1)). It follows easily that there are fixed components
$\Gamma_0,\Gamma_1,\Gamma_2$ of the three involutions $h$,
$(g,g)h(g^{-1},g^{-1})$, $(g^2,g^2)h(g^{-2},g^{-2})$ of $K$ respectively,
which intersect transversely at $p$ and have the same genus and
self-intersection. We claim that
$\Gamma_0,\Gamma_1,\Gamma_2$ are $(-2)$-spheres, and consequently
$(\sum_{k=0}^2\Gamma_k)^2=0$. To see that each $\Gamma_k$ is a $(-2)$-sphere,
it suffices to show that $\Gamma_k^2<0$ because $c_1(K_X)\cdot\Gamma_k=0$.
Suppose to the contrary that $\Gamma_k^2\geq 0$. Then
$(\sum_{k=0}^2\Gamma_k)^2>0$, which we will show is impossible
when $c_1(K_X)\neq 0$.
To see this, note that all three classes $\sum_{k=0}^2\Gamma_k$, $c_1(K_X)$,
and the symplectic structure $\omega$ are fixed under $K$. Since
$b_2^{+}(X/K)=1$, we may write
$$
\sum_{k=0}^2\Gamma_k=a_1\omega+\alpha_1, \;\; c_1(K_X)=a_2\omega+\alpha_2
$$
for some $a_1,a_2\in\R^{+}$ and $\alpha_1,\alpha_2\in H^2(X;\R)$ such that
$\alpha_i\cdot\omega=0$ and $\alpha_i^2\leq 0$ for $i=1,2$.
Without loss of generality we assume
that $\omega^2=1$. Then $(\sum_{k=0}^2\Gamma_k)^2>0$, $c_1(K_X)^2=0$,
and $c_1(K_X)\cdot\sum_{k=0}^2\Sigma_k=0$ give rise to
$$
a_1^2+\alpha_1^2>0, \; a_2^2+\alpha_2^2=0, \mbox{ and }
a_1a_2+\alpha_1\cdot\alpha_2=0.
$$
We arrive at a contradiction to the triangle inequality
$$
|\alpha_1\cdot\alpha_2|=a_1a_2>(\alpha_1^2\cdot\alpha_2^2)^{1/2}.
$$
Hence $\Gamma_0,\Gamma_1,\Gamma_2$ are $(-2)$-spheres and 
$(\sum_{k=0}^2\Gamma_k)^2=0$.

We claim that $\sum_{k=0}^2\Gamma_k=\lambda c_1(K_X)$ for some $\lambda>0$.
To see this, let $H^{+}$ be the space of self-dual harmonic $2$-forms.
Then since $b_2^{+}(X/K)=1$, the projections of the classes of
$\sum_{k=0}^2\Gamma_k$ and $c_1(K_X)$ into $H^{+}$ are linearly dependent.
On the other hand, $\sum_{k=0}^2\Gamma_k$ and $c_1(K_X)$ span an isotropic
subspace because
$$
(\sum_{k=0}^2\Gamma_k)^2=c_1(K_X)^2=c_1(K_X)\cdot\sum_{k=0}^2\Gamma_k=0,
$$
so that their projections into $H^{+}$ are injective. This proves the claim.

Now for each involution $h^\prime\in H$, the set
$h^\prime (\cup_{k=0}^2\Gamma_k)$ is disjoint from $\cup_{k=0}^2\Gamma_k$
because of positivity of intersection of $J$-holomorphic curves and
because
$$
(h^\prime)^\ast(\sum_{k=0}^2\Gamma_k)\cdot (\sum_{k=0}^2\Gamma_k)
=\lambda^2 (h^\prime)^\ast c_1(K_X)\cdot c_1(K_X)=\lambda^2 c_1(K_X)^2=0.
$$
Since there are $15$ distinct involutions in $H$, there must be $16$
such configurations as $\cup_{k=0}^2\Gamma_k$ which are mutually disjoint.
This certainly contradicts $b_2^{-}(X)=19$, and the lemma follows.
\end{proof}

If $c_1(K_X)=0$, then $X$ is already minimally exotic and we are done in
this case. Suppose $c_1(K_X)\neq 0$, then
with Lemma 2.6, we shall further argue that each of $(g,g)$ or
$(g,g^2)$ must have $6$ isolated fixed points. The reason is that if not,
there will be an increase to $\mu(H)\equiv |H|^{-1}\sum_{g\in H}tr(g)$, in
comparison with a symplectic automorphism group $H$ of a $K3$ surface,
of either $2\times \frac{6}{9}$ or $4\times \frac{6}{9}$, both of which are
not integral. (The centralizer of $(g,g)$ or $(g,g^2)$ is $\langle g\rangle\times \langle g\rangle$
which has order $9$, and $(g,g)$, $(g^2,g^2)$, and $(g,g^2)$, $(g^2,g)$
are not conjugate in $H$ even though each pair of them are conjugate in $G$.)
The proof for the case of $G=A_{4,4}$ is then completed.

\vspace{2mm}

End of Case (3) where $G=A_{4,4}$.

\vspace{2mm}

{\it Case (4)}. $G=T_{192}$ or $T_{48}$. Set $H\equiv [G,G]\subset G_0$.
Then in both cases, $\mu(H)=5$ for a symplectic automorphism group $H$ of 
a $K3$ surface (cf. \cite{Xiao}).

Let $G=T_{192}$. In this case $H=(Q_8\ast Q_8)\times_\phi\Z_3$, where
$$
Q_8\ast Q_8=Q_8\times Q_8/\langle (-1,-1)\rangle
$$
is the central product of $Q_8$ with itself, and the action of $\Z_3$
on $Q_8\ast Q_8$ is given by
$\phi: x\ast y\mapsto \alpha^{-1}(x)\ast\alpha(y)$ for some fixed order-$3$
automorphism $\alpha$ of $Q_8$ (cf. \cite{Mu}). The normalizer of
$\Z_3$ in $H$ is $\langle -1\rangle\times \Z_3$, where $\langle -1\rangle$ denotes the center of
$Q_8\ast Q_8$. It follows easily that for each $g\in\Z_3$, there are no
type (III) fixed points of $g$ because $b_2^{-}(X)=19$ and the index
of $\langle -1\rangle\times \Z_3$ in $H$ is $16$. By Lemma 2.3 (3), each order-$3$
element of $H$ has $6$ isolated fixed points, with possibly some fixed
toroidal components. Hence the case where $G=T_{192}$ follows.

Let $G=T_{48}$. Then $H$ is isomorphic to $T_{24}=Q_8\times_\phi\Z_3$.
By Lemma 2.3 (3), one only needs to verify that for any nontrivial element
$g\in\Z_3$, there are no groups of type (III) fixed points of $g$.

Suppose to the contrary that there is a group of type (III) fixed points,
which is contained in a type (C) component $\Lambda$. We observe that the
normalizer of $\Z_3$ in $H$ is $\langle -1\rangle\times \Z_3$ which has index $4$.
It follows immediately from $b_2^{-}(X)=19$ that $\Lambda$ is not
represented by a $\tilde{E}_7$ graph, or a $\tilde{D}_n$ graph with
$n>4$, or a $\tilde{A}_n$ graph with $n>2$. In fact $\Lambda$ is not
represented by a $\tilde{D}_4$ graph either. To see this, suppose $\Lambda$ is
of type $\tilde{D}_4$.  If $\Lambda$ is invariant under $-1\in Q_8$, then the $(-2)$-sphere 
represented by the central vertex of $\Lambda$ must contain $2$ fixed points of $-1$. 
On the other hand, by Lemma 2.2, $Q_8$ has at least $1$ fixed point (cf. e.g. \cite{Xiao}).
Since $-1\in Q_8$ has only $8$ isolated fixed points, it follows that there must be an order-$4$
element of $Q_8$ which also fixes the central vertex of the $\tilde{D}_4$ graph. However,
this would imply that the whole group $Q_8$ fixes the central vertex, and there is an induced
effective action of $Q_8$ on $\s^2$, which is a contradiction. If $\Lambda$ is 
not invariant under $-1\in Q_8$, then the orbit of $\Lambda$ under $H$ has at least 
$8$ components, contradicting $b_2^{-}(X)=19$. Hence $\Lambda$ is not represented 
by a $\tilde{D}_4$ graph.

It remains to eliminate the possibility that $\Lambda$ is represented by
a $\tilde{A}_2$ graph. Suppose this is the case. Then by the same argument
as above, $\Lambda$ can not be invariant under $-1\in Q_8$, which means that
$\Lambda$ comes in pairs. Furthermore, the constraint $b_2^{-}(X)=19$ allows
for exactly two $\tilde{A}_2$ components, which give $2$ groups of type (III)
fixed points of $g$. To eliminate this possibility, we make use of the fact
that there is an involution $h\in G\setminus H$, such that $hgh^{-1}=g^{-1}$.
There is an induced action of $h$ on the set of type (III) fixed points of $g$,
where by replacing $h$ with $(-1)h$, we may assume that $h$ fixes
the isolated fixed point in each of the $2$ groups of type (III) fixed
points. Since the local representation of $g$ at the fixed point is of type
$(1,1)$ or $(2,2)$, it follows that at the fixed point one has the
commutativity relation $hg=gh$, which contradicts the fact that $hgh^{-1}=g^{-1}$.
This finishes the proof that there are no groups of type (III) fixed points,
and the case where $G=T_{48}$ follows.

\section{Proof of Theorem 1.1}

In the proof of Theorem 1.1, the determination of the structure and the action
of the subgroup $G_0$ follows the strategy of Xiao \cite{Xiao}. However,
it relies on the fundamental work of Taubes \cite{T} to establish the
necessary properties of the action of $G$ in order to implement Xiao's
strategy.

The first half of Theorem 1.1 is contained in the following

\begin{proposition}
Let $X$ be a minimally exotic symplectic homotopy $K3$ surface, and let $G$
be a finite group acting on $X$ effectively and symplectically. Then there
exists a short exact sequence of finite groups
$$
1\rightarrow G_0\rightarrow G\rightarrow G^0\rightarrow 1,
$$
where $G^0$ is cyclic and $G_0$ is characterized as the maximal subgroup 
of $G$ with property $b_2^{+}(X/G_0)=3$.  Moreover, for each
$g\in G_0$ the action of $g$ is pseudofree with local representation
at a fixed point contained in $SL_2(\C)$, and the quotient orbifold
$X/G_0$ can be smoothly resolved into a minimally exotic symplectic 
homotopy $K3$ surface.
\end{proposition}

\begin{proof}
Let $\omega$ be a symplectic structure on $X$ which is preserved under
$G$, and we fix an $\omega$-compatible, $G$-equivariant almost complex
structure $J$ on $X$. Let $K_X$ be the canonical bundle with the choice
of $J$, and let $g_J$ be the associated Riemannian metric, both of which
are $G$-equivariant.

Following  Taubes \cite{T}, we consider the following family (parametrized by
$r>0$) of perturbed Seiberg-Witten equations
$$
D_A\psi=0 \mbox{ and } P_{+} F_A=\frac{1}{4}\tau(\psi\otimes\psi^\ast)+\mu,
$$
where $\psi=\sqrt{r}(\alpha,\beta)\in \Gamma(K_X\oplus\I)$, $A$ is a
$U(1)$-connection on $K_X$, and
$$
\mu=-\frac{ir}{4}\omega+ P_{+} F_{A_0}
$$
for a canonical (up to gauge equivalence) connection $A_0$ on $K_X^{-1}$.
According to \cite{T}, $c_1(K_X)$ is a Seiberg-Witten basic class, hence
for any $r>0$, there is a solution $(\psi, A)$ with
$\psi=\sqrt{r}(\alpha,\beta)\in \Gamma(K_X\oplus\I)$. Moreover, as
$r\rightarrow\infty$, the zero set $\alpha^{-1}(0)\subset X$ converges
pointwise to a set of finitely many $J$-holomorphic curves with multiplicity,
which represents the Poincar\'{e} dual of $c_1(K_X)$. Since $X$ is minimally exotic 
by our assumption, $c_1(K_X)=0$, and as a consequence $\alpha^{-1}(0)$ must be empty
for sufficiently large $r>0$. It follows that, when $r>0$ is sufficiently large, there is a 
unique solution $(\sqrt{r}(\alpha_0,0), A)$ (up to gauge equivalence) to the
perturbed Seiberg-Witten equations, where $a_0\equiv \frac{1}{2}(A-A_0)$
is a flat connection on $K_X$, $|\alpha_0|=1$ and $\nabla_{a_0}\alpha_0=0$
(cf. Lemma 4.5 and the proof of Proposition 4.4 in Taubes \cite{T}).

With the preceding understood, we note that since the family of
perturbed Seiberg-Witten equations under consideration is $G$-equivariant
($A_0$ may be chosen such that $g^\ast A_0=A_0$, $\forall g\in G$),
the uniqueness of  $(\alpha_0,a_0)$ up to gauge equivalence implies
that for any $g\in G$, $g^\ast\alpha_0=\phi(g)\alpha_0$ for some smooth
circle-valued function $\phi(g):X\rightarrow \s^1$. Since $g$ is of a finite
order, $\phi(g)$ must be a constant function because $\phi(g)^{|g|}=1$.
This gives rise to a homomorphism $\rho:G\rightarrow \s^1$ which is defined
by $\rho: g\mapsto \phi(g)\in\s^1$. We define $G_0\subset G$ to be the kernel
of $\rho$ and set $G^0\equiv G/G_0$. Then clearly $G^0$ is cyclic. Moreover,
if $g\in G$ has the property that $b_2^{+}(X/g)=3$ then, as we argued in
\cite{CK1}, the corresponding $g$-equivariant Seiberg-Witten invariant
is nonzero, which implies $\phi(g)=1$ and hence $g\in G_0$. Finally,
we observe that for any $g\in G_0$, since $\alpha_0$ is a nowhere vanishing
section of $K_X$ and $g^\ast\alpha_0=\alpha_0$,
$g$ has at most isolated fixed points with a local representation
contained in $SL_2(\C)$.

It remains to show that the quotient orbifold $X/G_0$ can be smoothly
resolved into a minimally exotic symplectic homotopy $K3$ surface. Note that
this automatically implies $b_2^{+}(X/G_0)=3$ as it equals the $b_2^{+}$
of the smooth resolution. In fact, in the next lemma we will prove an
equivariant version of it. To finish the proof of the proposition, one simply
uses the lemma with $H=K=G_0$. 

\end{proof}

Consider a subgroup $K$ of $G$ which is contained in $G_0=\ker \rho$
where $\rho:g\mapsto \phi(g)$, i.e., for any $g\in K$, $g^\ast\alpha_0=
\alpha_0$. Let $H$ be a normal subgroup of $K$.

\begin{lemma}
There exists a minimally exotic symplectic homotopy $K3$ surface $X_H$ which
is a smooth resolution of the orbifold $X/H$, such that $K/H$ acts on
$X_H$ symplectically, extending the natural $K/H$-action on $X/H$
under the resolution $X_H\rightarrow X/H$. Moreover, note that
$b_2^{+}(X_H/(K/H))=b_2^{+}(X/K)=b_2^{+}(X_K)=3$.
\end{lemma}

\begin{proof}
The construction of the smooth resolution of the symplectic orbifold
$X/H$ was given by McCarthy and Wolfson in \cite{McW}. We shall briefly
review the procedure, indicating that it can be done equivariantly.
In fact the construction is local, so we shall be focusing on a neighborhood
of an isolated singular point of the orbifold, which by the equivariant
Darboux' theorem is modeled on $\C^2/\Gamma$, where $\Gamma$ is the
isotropy group at the singular point which acts complex linearly on $\C^2$,
and where the symplectic structure $\omega_0$ on $\C^2/\Gamma$ is given 
by the standard one on $\C^2$, $\omega_{std}=i(dz_1\wedge d\bar{z}_1+dz_2\wedge d\bar{z}_2)$.

Let $U, V$ be the part of $\C^2/\Gamma$ which lies outside and inside of
the unit ball over $\Gamma$ respectively, and let $W=\partial U=\partial V$
which is the $3$-manifold $\s^3/\Gamma$. Since $V$ is an algebraic surface
with an isolated singularity, there is a nonsingular, minimal projective
resolution $\pi:Y\rightarrow V$. Note that $Y$ is K\"{a}hler. We let
$\tau$ be a K\"{a}hler form on $Y$. Then for any $\epsilon>0$,
$\omega_\epsilon\equiv \pi^\ast\omega_0+\epsilon \tau$ is a K\"{a}hler 
form on $Y$. We shall show that for a sufficiently small $\epsilon>0$,
the two pieces $(U,\omega_0)$ and $(Y,\omega_\epsilon)$ can be symplectically
``glued'' together, which gives a smooth resolution of $\C^2/\Gamma$ by
a symplectic manifold.

To this end, we consider the contact structure $\xi$ on $W$ which is the
distribution of complex lines in $TW$. Note that $\omega_0|_W=d\alpha$ for
some contact form $\alpha$ such that $\xi=\ker\alpha$. On the other hand,
since $W$ is a rational homology $3$-sphere, $\tau|_W=d\beta$ for a $1$-form
$\beta$, and hence $\omega_\epsilon|_W=d\alpha_\epsilon$ where
$\alpha_\epsilon\equiv \alpha+\epsilon\beta$ is also a contact form when
$\epsilon>0$ is sufficiently small. By Moser's argument, there exists a
self-diffeomorphism $\psi:W\rightarrow W$ such that $\psi^\ast\alpha_\epsilon
=e^f \alpha$ for some smooth function $f:W\rightarrow \R$. Pick a constant
$C>0$ such that $f<C$ on $W$. Let $Z\subset (\R\times W, d(e^t\alpha))$ be
the symplectic ``cylinder'' defined by
$$
Z\equiv \{(t,x)|x\in W, f(x)-C\leq t\leq 0\}.
$$
Then the smooth resolution of $\C^2/\Gamma$ by a symplectic manifold is
given by
$$
(X_{\epsilon,C},\omega)\equiv
(U,\omega_0)\cup (Z, d(e^t\alpha))\cup (Y, e^{-C}\omega_\epsilon),
$$
where the gluing between $\partial U=W$ and the component of $\partial Z$
defined by $t=0$ is by the identity map on $W$, and the gluing between the
component of $\partial Z$ defined by $t=f(x)-C$ and $\partial Y=W$ is by
$(t,x)\mapsto \psi(x)$, where $\psi:W\rightarrow W$ is the self-diffeomorphism
obtained above through Moser's argument. We leave it to the reader to follow
through that if a finite group $\Gamma^\prime$ acts complex linearly on
$\C^2/\Gamma$, then there is a corresponding symplectic
$\Gamma^\prime$-action on the smooth resolution $(X_{\epsilon,C},\omega)$.
(We remark that Moser's argument can be done equivariantly in the presence
of a compact Lie group action; in particular, the self-diffeomorphism
$\psi$ of
$W$ can be made equivariant with respect to the $\Gamma^\prime$-action on
$W$, so that the gluing by $(t,x)\mapsto \psi(x)$ in the construction
of $X_{\epsilon,C}$ is also equivariant.)

It remains to show that $X_H$ is a minimally exotic symplectic homotopy $K3$
surface, and that $b_2^{+}(X_H/(K/H))=3$. The key step is the observation
that $X_H$ has a trivial canonical bundle. To see this, note that for any
$g\in H$, since $g^\ast\alpha_0=\alpha_0$, the nonzero section $\alpha_0$
descends to a nonzero section $\hat{\alpha_0}$ of the canonical bundle
of the symplectic orbifold $X/H$. With this understood it suffices to show that 
the canonical bundle of $(X_{\epsilon,C},\omega)$ is trivial, which is done by 
matching up the trivialization of the canonical bundle on the three pieces 
$(U,\omega_0)$, $(Z, d(e^t\alpha))$  and $(Y, e^{-C}\omega_\epsilon)$. 

On $(U,\omega_0)$, the canonical bundle
$K_U$ is trivialized by $\hat{\alpha_0}$.
On $(Z, d(e^t\alpha))$, the canonical bundle is the pull back of $\xi^{-1}$,
the inverse line bundle of the contact structure $\xi$, via the projection
$Z\rightarrow W$. Since $K_U|_W=\xi^{-1}$ and $K_U$ is trivial, we see
that $K_Z$ is also trivial. Finally, $K_Y$ is also
trivial, because the symplectic form $e^{-C}\omega_\epsilon$ on $Y$ is
K\"{a}hler so that $K_Y$ is simply given by the holomorphic canonical
bundle. Since for each $g\in H$ the local representation at each
fixed point of $g$ is contained in $SL_2(\C)$, the singularity of
$\C^2/\Gamma$ is a Du Val singularity, and it is known that in this case
$Y$ has a trivial canonical bundle if it is taken minimal. Now since
$H^1(W;\Z)=0$ which parametrizes the homotopy classes of the non-zero
sections of the trivial line bundle over $W$, one can matches up the
trivialization of $K_U$, $K_Z$ and $K_Y$ to obtain a trivialization for
the canonical bundle of $(X_{\epsilon,C},\omega)$.

As an immediate consequence, $X_H$ is spin as $w_2(TX_H)=c_1(K_{X_H})\pmod{2}$
must vanish. By Rohlin's theorem, $\text{sign}(X_H)$ is divisible by $16$.
Hence the intersection form on $H_2(X_H;\Z)/Tor$ is given by
$m \left(\begin{array}{cc}
0 & 1\\
1 & 0\\
\end{array} \right )
\oplus 2k (\pm E_8)$, with $m=b_2^{+}(X_H)$ and $k=|\text{sign}(X_H)|/16$.
Now observe that the fundamental group of $X_H$ is finite, which
implies that $b_1(X_H)=0$. Hence
$$
0=c_1^2(K_{X_H})=(2\chi+3\text{sign})(X_H)=2(2+2m+16k)\pm 3\cdot 16k=0.
$$
Since $m=b_2^{+}(X_H)=b_2^{+}(X/H)=1$ or $3$ (cf. Lemma 2.1), the only
solution for $(m,k)$ from the above equation is $m=3$ and $k=1$, and
moreover, $\text{sign}(X_H)=-16$. This shows that $X_H$ is a rational
homology $K3$ surface. (Note that this conclusion also follows directly
from Furuta's work on the $\frac{11}{8}$-conjecture, cf. \cite{Fu}.)

Next we show that $\pi_1(X_H)$ is trivial. Let $\widehat{X_H}$ be the
universal cover of $X_H$, which is compact because $\pi_1(X_H)$ is finite.
Then $\widehat{X_H}$ is a closed, simply-connected symplectic $4$-manifold
with trivial canonical bundle. It was shown by Morgan and Szab\'{o} \cite{MS}
(compare also \cite{Li}) that the Betti numbers of $\widehat{X_H}$ satisfy
$$
b_2^{+}(\widehat{X_H})=3 \mbox{ and } b_2^{-}(\widehat{X_H})=19.
$$
With $b_2^{+}(X_H)=3$ and $b_2^{-}(X_H)=19$ it follows easily that
$\pi_1(X_H)$ is trivial. This completes the proof
that $X_H$ is a minimally exotic symplectic homotopy $K3$ surface.

Finally, we observe that
$$
b_2^{+}(X_H/(K/H))=b_2^{+}((X/H)/(K/H))=b_2^{+}(X/K)=b_2^{+}(X_K)=3.
$$
\end{proof}

\begin{remark}
The holomorphic version of Lemma 3.2 has been used in a fundamental way,
first by Nikulin in \cite{N} and then by Xiao in \cite{Xiao}, to study
finite symplectic automorphism groups of $K3$ surfaces. In particular,
following the argument in Nikulin \cite{N}, one can show, with Lemma 3.2,
that for any $g\in G_0$, the order $|g|\leq 8$ and the number of fixed
points of $g$ is the same as that of an order $|g|$ symplectic automorphism
of a $K3$ surface. However, we would like to point out that this statement
can also be proved directly, by a lengthy argument involving essentially
the various $G$-index theorems. Even though we have no need to pursue it
here, we would like to observe that $\text{Fix}(g)\neq\emptyset$
directly implies that the smooth resolution $X_H$ in Lemma 3.2 is
simply-connected, without appealing to the result of Morgan and Szab\'{o}
in \cite{MS} as we did in the proof of Lemma 3.2.
\end{remark}

Now with Lemma 3.2 in place, we shall follow through the arguments of Xiao
in \cite{Xiao} to complete the proof of Theorem 1.1 by showing that $G_0$
is a symplectic $K3$ group and that the action of $G_0$ on $X$ has the same
fixed point set structure as does a corresponding symplectic automorphism
group of a $K3$ surface.

In Section $1$ of Xiao \cite{Xiao}, the only argument involving complex
geometry is in the proof of Lemma $2$ there. We shall give
a pure algebraic topology proof of this result below. In order to state
the lemma, we first need to introduce the necessary notations.

Let $X$ be a minimally exotic symplectic homotopy $K3$ surface and let $G$ be
a finite group acting effectively on $X$ via symplectic symmetries such that
$b_2^{+}(X/G)=3$. Then as we have shown, $X/G$ is a symplectic orbifold
of only Du Val singularities, which has a smooth resolution $X_G$ as
defined in Lemma 3.2. Let $L^\prime$ be the sublattice of $H_2(X_G;\Z)$
generated by the $(-2)$-spheres in $X_G$ which are sent to the singular
points under $X_G\rightarrow X/G$, and let $L$ be the smallest primitive
sublattice of $H_2(X_G;\Z)$ containing $L^\prime$. Then the analog of Lemma
$2$ in Xiao \cite{Xiao} is contained in the following lemma.

\begin{lemma}
$L/L^\prime\cong G/[G,G]$.
\end{lemma}

\begin{proof}
Let $A$ be a regular neighborhood of the $(-2)$-spheres in $X_G$ which
are mapped to the singular points under $X_G\rightarrow X/G$, and let
$B=X_G\setminus A$ be the complement of $A$. Then the long exact sequence
associated to the pair $(X_G,A)$ gives rise to
$$
H_2(A;\Z)\stackrel{i_\ast}{\rightarrow} H_2(X_G;\Z)
\stackrel{j_\ast}{\rightarrow} H^2(B;\Z)\rightarrow 0,
$$
where we have used the excision and Poincar\'{e} duality to make the
identification $H_2(X_G,A;\Z)\cong H_2(B,\partial B;\Z)\cong H^2(B;\Z)$,
and we have used the fact that $A$ is simply-connected so that
$H_1(A;\Z)=0$. On the other hand, by the universal-coefficient theorem
for cohomology, we have the short exact sequence
$$
0\rightarrow \text{Ext}(H_1(B;\Z),\Z)\rightarrow H^2(B;\Z)
\stackrel{h}{\rightarrow} \text{Hom}(H_2(B;\Z),\Z)\rightarrow 0.
$$
Now observe that for any element $x\in H_2(X_G;\Z)$, $h\circ j_\ast(x)=0$
if and only if the intersection product of $x$ with any element $y\in
H_2(B;\Z)$ is zero, which is precisely if and only if $x\in L$. This
gives a surjective homomorphism
$j_\ast: L\rightarrow \text{Ext}(H_1(B;\Z),\Z)$ whose kernel is easily seen
to be $L^\prime=\text{Im}(i_\ast)$. Hence
$L/L^\prime\cong  \text{Ext}(H_1(B;\Z),\Z)$.

Finally, $\pi_1(B)=G$ so that $H_1(B;\Z)=G/[G,G]$ is a torsion group.
This gives
$$
L/L^\prime\cong \text{Ext}(H_1(B;\Z),\Z)\cong H_1(B;\Z)=G/[G,G].
$$
\end{proof}

In Section $2$ of \cite{Xiao}, Xiao formulated a set of criteria
obtained from Section $1$, and by a computer search a list of
possibilities for a symplectic $K3$ group as well as the combinatorial
types of the actions were generated. A few of the cases were further
eliminated to reach the final list, where the arguments are those
in \cite{Xiao} which precedes Lemma $5$. We observe that this procedure 
can be used in the present situation verbatim. This finishes the proof of 
Theorem 1.1.

\begin{remark}
The holomorphic version of Theorem 1.1 is contained in Nikulin \cite{N}.
There it was also shown that the order of the cyclic group $G^0$ is bounded
by $66$ (which is a sharp bound). The proof of this result involves
arguments in complex geometry which are not available in the present,
symplectic category. However, we should point out that there are further
informations contained in the proof of Proposition 3.1 which can be
used to analyze $G^0$; in particular, it is very likely
that $|G^0|$ has a universal upper bound. We shall not pursue this issue
here, but wish to point out that because of the homological rigidity of
symplectic symmetries of a minimally exotic symplectic homotopy $K3$ surface
established in \cite{CK1}, the prime factors in $|G^0|$ are bounded by
$b_2=22$.

\end{remark}

\section{The Lattice $L_X$ and Proof of Theorem 1.2}

Recall that the Seiberg-Witten invariant of a simply-connected, closed,
oriented, smooth $4$-manifold $M$ with $b_2^{+}\geq 2$ is a map
$$
SW_M: \{\beta\in H^2(M;\Z)|\beta\equiv w_2(TM)\pmod{2}\}\rightarrow \Z.
$$
A class $\beta$ is called a (Seiberg-Witten) basic class if
$SW_M(\beta)\neq 0$. It is a fundamental fact that the set of basic classes
is finite. Moreover, if $\beta$ is a basic class, then so is $-\beta$ with
$$
SW_M(-\beta)=(-1)^{(\chi+\text{sign})(M)/4} SW_M(\beta).
$$
When $M$ is symplectic, a fundamental result of Taubes says that the canonical
class $c_1(K_X)$ associated to a symplectic structure is always a basic class.
The Seiberg-Witten invariant $SW_M$ is an invariant of the diffeomorphism
class of $M$, whose sign depends on a choice of an orientation of
$$
H^0(M;\R)\otimes \det H^{2,+}(M;\R).
$$
In particular, the set of basic classes depends
only on the diffeomorphism type of $M$. When $M$ is a homotopy $K3$ surface,
a theorem of Morgan and Szab\'{o} \cite{MS} says that $\beta=0$ is always a
basic class. Furthermore, when $M$ is symplectic, work of
Taubes \cite{T} gives additional information about the Seiberg-Witten
invariant, in particular, about the set of basic classes.

Let $X$ be a symplectic homotopy $K3$ surface. We set
$$
L_X\equiv \text{Span}(\beta\in H^2(X;\Z)|SW_X(\beta)\neq 0)\subset H^2(X;\Z),
$$
and set $r_X\equiv \text{rank}(L_X)$. Let $\omega$ be any symplectic
structure on $X$, and let $K_X$ be the associated canonical bundle. Then
Taubes \cite{T} showed that $c_1(K_X)\in L_X$ and
$0\leq \beta\cdot [\omega]\leq c_1(K_X)\cdot [\omega]$ for any basic
class $\beta$. In particular, $c_1(K_X)=0$ iff $r_X=0$.

\begin{proposition}
Let $X$ be a symplectic homotopy $K3$ surface. Then the lattice of basic
classes $L_X$ is isotropic, i.e., for any $x,y\in L_X$, the cup product
of $x$ and $y$ is zero. As a consequence, the rank of $L_X$ is bounded
from above by $\min(b_2^{+},b_2^{-})=3$, i.e., $r_X\leq 3$.
\end{proposition}

\begin{proof}
Let $\omega$ be a symplectic structure of $X$, and let $K_X$ be the canonical
bundle. Since $X$ is minimal, and $c_1^2(K_X)=2\chi(X)+3\text{sign}(X)=0$,
a theorem of Taubes (cf. \cite{T}, Theorem 0.2 (5)) says that for any basic
class $\beta$, $e_\beta\equiv \frac{1}{2}(c_1(K_X)+\beta)\in H^2(X;\Z)$ is
Poincar\'{e} dual to $\sum_i m_i T_i$, where $m_i>0$ and $\{T_i\}$ is a
finite set of disjoint, symplectically embedded tori with self-intersection
$0$.

To see $L_X$ is isotropic, it suffices to show that for any basic classes
$\beta,\beta^\prime$, the cup product $\beta\cdot \beta^\prime=0$, which
follows from the generalized adjunction formula as follows. Suppose
$e_\beta=\sum_i m_i T_i$ where $\{T_i\}$ is a finite set of disjoint,
symplectically embedded tori with self-intersection $0$. Then for any
basic class $\beta^\prime$, we apply the generalized adjunction formula 
to $T_i$, 
$$
\text{genus}(T_i)\geq 1+\frac{1}{2}(|\beta^\prime \cdot T_i|+T_i^2).
$$
This implies, for each $i$, $\beta^\prime\cdot T_i=0$ because
$\text{genus}(T_i)=1$ and $T_i^2=0$, and consequently,
$e_\beta\cdot \beta^\prime=0$.
In particular, since $c_1(K_X)$ is a basic class, we have $e_\beta\cdot
c_1(K_X)=0$, which implies that $\beta\cdot c_1(K_X)=0$ for any basic
class $\beta$. (This is because $e_\beta\equiv \frac{1}{2}(c_1(K_X)+\beta)$
and $c_1^2(K_X)=0$.)
Now we go back to $e_\beta\cdot \beta^\prime=0$, and
conclude that
$$
\beta\cdot\beta^\prime=2e_\beta\cdot\beta^\prime-c_1(K_X)\cdot\beta^\prime=0.
$$

Finally, we point out that $r_X\leq 3$ follows directly from the fact that
the projection of $L_X$ into $H^{2,+}(X;\Z)$ is injective (because $L_X$ is
isotropic).
\end{proof}

\begin{remark}
Suppose $G$ is a finite group which acts on a symplectic homotopy $K3$ surface
$X$ via symplectic symmetries. Then there is an induced action of $G$ on the
set of basic classes, which can be extended to a linear action on the lattice
$L_X$. Moreover, let $\omega$ be the symplectic structure which is preserved
under the action of $G$, and let $K_X$ be the associated canonical bundle.
Then $c_1(K_X)\in L_X$ is fixed under the action of $G$, and since $\omega$
is also fixed, the function $\omega: L_X\rightarrow \R$ defined by
$x\mapsto [\omega]\cdot x$ is $G$-invariant. On the other hand, since
the action of $G$ on $H^0(X;\R)\otimes \det H^{2,+}(X;\R)$ is
orientation-preserving (cf. Lemma 2.1), one has, for any basic class $\beta$,
$$
SW_X(g\cdot\beta)=SW_X(\beta), \;\;\forall g\in G.
$$
It is clear that the induced action of $G$ on $L_X$ may be exploited to
relate the action of $G$ on $X$ and the underlying smooth structure of $X$.

\end{remark}

\noindent{\bf Proof of Theorem 1.2}

\vspace{3mm}

Construction of this type of exotic $K3$ surfaces is due to Fintushel
and Stern, which is done by performing the knot surgery on three disjoint,
homologically distinct, symplectically embedded tori in a Kummer surface
(cf. \cite{FS}, compare also \cite{GM}). Our observation here is that it
can be done equivariantly. However, we would like to point out that the
three tori (actually $12$ tori divided into $3$ groups) have to be chosen
differently than in \cite{FS} and \cite{GM} (cf. Remark 4.3).

Consider the $4$-torus $T^4=(\s^1)^4$ with the involution $\rho$, which is
defined in the angular coordinates by
$$
\rho: (\theta_0, \theta_1,\theta_2,\theta_3)\mapsto
(-\theta_0, -\theta_1,-\theta_2,-\theta_3), \mbox{  where }
\theta_j\in \R/2\pi\Z.
$$
There are $16$ isolated fixed points $(\theta_0, \theta_1,\theta_2,\theta_3)$
where each $\theta_j$ takes values in $\{0,\pi\}$. A Kummer surface is
a smooth $4$-manifold which is obtained by replacing each of the singular
points in the quotient $T^4/\langle \rho\rangle$ with an embedded $(-2)$-sphere.
We denote the $4$-manifold by $X_0$.

We shall give a more concrete description of $X_0$ below, where $X_0$ is
also naturally endowed with a symplectic structure.
Consider the symplectic form $\Omega$ on
$T^4$, which is equivariant with respect to the involution $\rho$:
$$
\Omega\equiv \sum_{(i,j,k)}(d\theta_0\wedge d\theta_i+d\theta_j\wedge
d\theta_k)
$$
where the sum is taken over $(i,j,k)=(1,2,3), (2,3,1)$ and $(3,1,2)$.
This gives rise to a symplectic structure on the orbifold $T^4/\langle\rho\rangle$.
One can further symplectically resolve the orbifold singularities to
obtain a symplectic structure on $X_0$ as follows. By the equivariant
Darboux' theorem, the symplectic structure is standard near each
orbifold singularity. In particular, it is modeled on a neighborhood of
the origin in $\C^2/\{\pm 1\}$ and admits a Hamiltonian $\s^1$-action with
moment map $\mu:(w_1,w_2)\mapsto\frac{1}{4}(|w_1|^2+|w_2|^2)$, where $w_1,w_2$
are the standard coordinates on $\C^2$. Fix a sufficiently small $r>0$
and remove $\mu^{-1}([0,r))$ from $T^4/\langle\rho\rangle$ at each of its singular point.
Then $X_0$ is diffeomorphic to the $4$-manifold obtained by collapsing
each orbit of the Hamiltonian $\s^1$-action on the boundaries $\mu^{-1}(r)$,
which is naturally a symplectic $4$-manifold (cf. \cite{Le}). We denote the
symplectic structure on $X_0$ by $\omega_0$.

Let $G=(\Z_2)^3$. We shall next describe a $G$-action on $X_0$ which
preserves the symplectic structure $\omega_0$. Consider first the following
$G$-action on $T^4$:
$$
a\cdot (\theta_0,\theta_1,\theta_2,\theta_3)
=(\theta_0, \theta_1+\pi a_1,\theta_2+\pi a_2,\theta_3+\pi a_3)
$$
where $a=(a_1,a_2,a_3)\in G$ with each $a_j\in\Z_2\equiv\Z/2\Z$. One can
check easily that the above $G$-action commutes with the involution $\rho$,
so that there is an induced $G$-action on the orbifold $T^4/\langle\rho\rangle$. Moreover,
the $G$-action clearly preserves the symplectic form
$$
\Omega\equiv \sum_{(i,j,k)}(d\theta_0\wedge d\theta_i+d\theta_j\wedge
d\theta_k)
$$
on $T^4$, hence it descends to a symplectic $G$-action on $T^4/\langle\rho\rangle$. From
the description of $(X_0,\omega_0)$ given in the previous paragraph, it
follows easily that there is an induced, symplectic $G$-action on
$(X_0,\omega_0)$. (The key point here is that Lerman's symplectic cutting
can be done equivariantly, cf. \cite{Le}.) The $G$-action on $X_0$ is
pseudofree; in fact, for any $0\neq a=(a_1,a_2,a_3)\in G$, a fixed point
of $a$ in $X_0$ is the image of a point in $T^4$ with angular coordinates
$(\theta_0,\theta_1,\theta_2,\theta_3)$, where $\theta_0=0$ or $\pi$,
and for $j=1,2,3$, $\theta_j=0$ or $\pi$ if $a_j=0$ and $\theta_j=\pi/2$
or $3\pi/2$ if $a_j=1$. (Note that each $a\neq 0$ in $G$ has $8$ isolated
fixed points.)

We shall next describe a set of $12$ disjoint, symplectically embedded tori
in $(X_0,\omega_0)$, which is invariant under the $G$-action. The $12$ tori
are divided into $3$ groups, labeled naturally by $j=1,2,3$, where each group
consists of $4$ tori. The group $G$ acts freely and transitively among the $4$ tori
in each group. For simplicity we shall only describe the group of tori indexed by $j=1$
in detail; the others are completely parallel.

Consider the projection $\pi_1$ from $T^4$ to $T^2$ where
$$
\pi_1:(\theta_0,\theta_1,\theta_2,\theta_3)\mapsto (\theta_2,\theta_3).
$$
For any fixed $\delta_{12},\delta_{13}\in \R/2\pi\Z$ other than $0$, $\pi/2$,
$3\pi/2$ and $\pi$, the $4$ tori in $T^4$
\begin{eqnarray*}
T_{1,0}\equiv\pi_1^{-1}(\delta_{12},\delta_{13}) &
T_{1,1}\equiv\pi_1^{-1}(\delta_{12}+\pi,\delta_{13})\\
T_{1,2}\equiv\pi_1^{-1}(\delta_{12},\delta_{13}+\pi) &
T_{1,3}\equiv\pi_1^{-1}(\delta_{12}+\pi,\delta_{13}+\pi)
\end{eqnarray*}
are symplectically embedded with respect to the symplectic form
$$
\Omega\equiv\sum_{(i,j,k)}(d\theta_0\wedge d\theta_i+d\theta_j\wedge
d\theta_k).
$$
Moreover, they descend to $4$ disjoint tori in $T^4/\langle\rho\rangle$, and if the
distance between $\delta_{12}$, $\delta_{13}$ to $0$, $\pi/2$,
$3\pi/2$ and $\pi$ is sufficiently large, $T_{1,k}$, $0\leq k\leq 3$,
can be regarded as tori in $X_0$, which are disjoint and symplectically
embedded. The union $\cup_k T_{1,k}$ is easily seen to be invariant under
the action of $G$ on $X_0$. Moreover,
the action of $G$ on $\cup_k T_{1,k}$ is
transitive, and each $T_{1,k}$ is invariant under an involution of $G$,
which acts on the torus freely via translations.

In the same vein, one can consider projections
$$
\pi_2:(\theta_0,\theta_1,\theta_2,\theta_3)\mapsto (\theta_1,\theta_3)
\mbox{ and }
\pi_3:(\theta_0,\theta_1,\theta_2,\theta_3)\mapsto (\theta_1,\theta_2)
$$
and choose $\delta_{21},\delta_{23},\delta_{31},\delta_{32}\in\R/2\pi\Z
\setminus \{0,\pi/2,\pi,3\pi/2\}$ to obtain $8$ other tori $T_{j,k}$,
where $j=2,3$ and $0\leq k\leq 3$. One can check easily that under
further conditions:
$$
\delta_{13}\neq \pm\delta_{23},\pm(\delta_{23}+\pi),\;
\delta_{12}\neq \pm\delta_{32},\pm(\delta_{32}+\pi),\;
\delta_{21}\neq \pm\delta_{31},\pm(\delta_{31}+\pi)\;
$$
the $12$ tori $T_{j,k}$ in $X_0$ are disjoint.

The exotic $K3$ surfaces are constructed by performing the Fintushel-Stern
knot surgery on each of the $12$ tori $T_{j,k}$ in $X_0$ with a fibered
knot. The key issue here is that the knot surgery needs to be performed
equivariantly with respect to the $G$-action on $X_0$. To this end, we
shall first give a brief review of the knot surgery from \cite{FS}.

Let $M$ be a simply-connected smooth $4$-manifold with $b_2^{+}>1$, and let
$T$ be a c-embedded torus in $M$ (cf. \cite{FS}) such that $\pi_1(M\setminus T)=1$.
Consider a knot $K$ in $\s^3$, and let $m$ denote
a meridional circle to $K$. Let $Y_K$ be the $3$-manifold obtained by
performing $0$-framed surgery on $K$. Then $m$ can also be viewed as a
circle in $Y_K$. In $Y_K\times\s^1$ we have the smoothly embedded torus
$T_m\equiv m\times\s^1$ of self-intersection $0$. Since a neighborhood
of $m$ has a canonical framing in $Y_K$, a neighborhood of the torus $T_m$
in $Y_K\times\s^1$ has a canonical identification with $T_m\times D^2$.
With this understood, the knot surgery on $T$ with knot $K$ is the
smooth $4$-manifold $M_K$, which is the fiber sum
$$
M_K\equiv M\#_{T=T_m}(Y_K\times\s^1)=[M\setminus (T\times D^2)]\cup
[(Y_K\times\s^1)\setminus (T_m\times D^2)].
$$
Here $T\times D^2$ is a tubular neighborhood of the torus $T$ in $M$.
The two pieces are glued together so as to preserve the homology class
$[\text{pt}\times \partial D^2]$. Note that the diffeomorphism type of
the fiber sum is not uniquely determined in general, and the $4$-manifold
$M_K$ is taken to be any manifold constructed in this fashion. A fundamental 
theorem of Fintushel and Stern states that $M_K$ is naturally homeomorphic
to $M$ and the Seiberg-Witten invariants of the two manifolds are related
by 
$$
sw_{M_K}=sw_{M}\cdot \Delta_K(t),
$$ 
where $sw_{M_K}$, $sw_{M}$ are certain Laurent polynomials defined from 
the Seiberg-Witten invariants of $M_K$ and $M$ respectively, and $\Delta_K(t)$ 
is the Alexander polynomial of $K$, with $t=\exp(2[T])$. See \cite{FS} for more details.
We remark that when $M$ is symplectic and $T$ is symplectically embedded,
$M_K$ can be naturally made symplectic by choosing any fibered knot $K$.
Note that when $M$ is the standard $K3$ surface, one has $sw_M=1$, so that
$M_K$ is an exotic $K3$ surface as long as the knot $K$ has a nontrivial
Alexander polynomial.

With the preceding understood, note that in our present situation, each of
the $12$ tori $T_{j,k}$ is invariant under an involution of $G$. Moreover,
the action on the tubular neighborhood $T_{j,k}\times D^2$ projects to
a trivial action on the $D^2$-factor. In order to do the knot surgery
equivariantly, we shall consider the involution on $Y_K\times\s^1$ which is
trivial on the $Y_K$-factor and is by translation on the $\s^1$-factor.
Recall that the only requirement in the knot surgery is to preserve the
homology class $[\text{pt}\times \partial D^2]$ under the gluing. Since
on the $Y_K\times\s^1$ side $\text{pt}\times \partial D^2$ is given by a
$0$-framed copy of the knot $K$ in $Y_K$ and the involution on $Y_K\times
\s^1$ is chosen to be trivial on the $Y_K$-factor, it follows easily that
for any fixed fibered knot $K$, one can do the knot surgery simultaneously
on each of the $12$ tori $T_{j,k}$ with the knot $K$, such that the
$G$-action on $X_0$ can be extended to a symplectic $G$-action on the
resulting $4$-manifold $X_K$. Moreover, $X_K$ continues to be simply connected
as repeated knot surgeries on parallel copies is equivalent to a single knot surgery 
using the connected sum of the knots (cf. Example 1.3 in \cite{C1}). Hence $X_K$ 
is a symplectic homotopy $K3$ surface with 
$$
sw_{X_K}=\Delta_K(t_1)^4\Delta_K(t_2)^4\Delta_K(t_3)^4
$$
where $t_j=\exp(2[T_{j,0}])$. Note that the three tori $T_{1,0}$, $T_{2,0}$
and $T_{3,0}$ are homologically linearly independent, so that $X_K$ is maximally exotic, 
i.e., $r_X=3$. By nature of construction, the $G$-action on $X_K$ is clearly pseudofree 
and induces a trivial action on the lattice $L_X$ of the Seiberg-Witten basic classes.

Finally, one obtains infinitely many distinct $X_K$ by choosing $K$ with distinct genus.

\hfill $\Box$

\begin{remark}
We would like to explain why the tori in our construction have to be
chosen differently than in \cite{FS} or \cite{GM}, and point out
that for the same reason our construction can not be extended to the
symplectic $K3$ group $(\Z_2)^4$. The key point here is that one has
to make sure that each $T_{j,k}$ can only be invariant under a cyclic
subgroup of $G$. Otherwise, we will be forced to introduce a nontrivial
cyclic action on the factor $Y_K$ in $Y_K\times \s^1$. Of course,
one way to obtain such a cyclic action on $Y_K$
is to pick a cyclic action on $\s^3$ under which $K$ is invariant, and
then do the $0$-framed surgery on $K$ equivariantly. The problem is that
the action on the tubular neighborhood $T_{j,k}\times D^2$ projects to
a trivial action on the $D^2$-factor, and since under the knot surgery
$\text{pt}\times \partial D^2$ is glued to a $0$-framed copy of $K$ in
$Y_K$, the action on $\s^3$ which we picked at the beginning has to fix
the knot $K$. However, by the Smith conjecture \cite{MH}
this is not possible unless
$K$ is the unknot. With this understood, we remark that with the choice
of tori as in \cite{FS} or \cite{GM}, one can only construct a
$(\Z_2)^2$-action on a homotopy $K3$ surface with maximal exoticness.
On the other hand, for the group $G=(\Z_2)^4$, our construction would
not even yield an effective $G$-action on a homotopy $K3$ surface
with nontrivial exoticness (i.e. $r_X>0$).

\end{remark}

\noindent{\bf Acknowledgments}: The authors wish to thank an anonymous referee whose
extensive comments have led to a much improved exposition.

\vspace{2mm}

{\Small
W. Chen: Department of Math. and Stat., University of Massachusetts,
Amherst, MA 01003.\\
{\it e-mail:} wchen@math.umass.edu

S. Kwasik: Mathematics Department, Tulane University,
New Orleans, LA 70118. \\
{\it e-mail:} kwasik@math.tulane.edu\\
}


\vspace{3mm}


\begin{thebibliography}{}

\bibitem {AB} M.F. Atiyah and R. Bott, {\em A Lefschetz fixed point formula
                        for elliptic complexes: II. Applications}, Ann. of
                        Math. {\bf 88} (1968), 451-491.

\bibitem {AH} M.F. Atiyah and F. Hirzebruch, {\em Spin-manifolds and group
                      actions}, Essays on Topology and Related Topics
                      (M\'{e}moires d\'{e}di\'{e}s \`{a} Georges de Rham),
                       pp. 18-28, Springer, New York, 1970.


\bibitem {Bry} J. Bryan, {\em Seiberg-Witten theory and $\Z/2^p$ actions
                          on spin $4$-manifolds}, Math. Res. Lett. {\bf 5}
                          (1998), no. 1-2, 165-183.

\bibitem {C} W. Chen, {\em Smooth $s$-cobordisms of elliptic $3$-manifolds},
                        Journal of Differential Geometry {\bf 73} no.3 (2006),
                        413-490.
                        
\bibitem {C1} ----------, {\em On the orders of periodic diffeomorphisms of 4-manifolds},
                                      Duke Math. J. (in press). 
                                      
\bibitem {CK1} W. Chen and S. Kwasik, {\em Symplectic symmetries of
                                   $4$-manifolds}, Topology {\bf 46} no.2
                                   (2007), 103-128.

\bibitem {CK2} -----------, {\em Symmetries and exotic
                                     smooth structures on
                                     a $K3$ surface}, J. Topology {\bf 1} (2008), 923-962.

\bibitem {Con} J.H. Conway, et al., {\em Atlas of Finite Groups},
                           Oxford Univ. Press, 1985.

\bibitem {FS} R. Fintushel and R. Stern, {\em Knots, links, and
                                  $4$-manifolds}, Invent. Math. {\bf 134}
                                  (1998), 363-400.

\bibitem {Fu} M. Furuta, {\em Monopole equation and the
                    $\frac{11}{8}$-conjecture}, Math. Res. Lett. {\bf 8}(2001),
                    279-291


\bibitem {GM} R. Gompf and T. Mrowka, {\em Irreducible $4$-manifolds need
                        not be complex}, Ann. Math. {\bf 138} (1993), 61-111

\bibitem {HZ} F. Hirzebruch and D. Zagier, {\em The Atiyah-Singer Theorem
                                and Elementary Number Theory}, Math.
                                Lecture Series {\bf 3}, Publish or Perish,
                                Inc., 1974

\bibitem {Hsiang} W.Y. Hsiang, {\em On the bound of the dimensions of
                            the isometry groups of all possible Riemannian
                            metrics on an exotic sphere}, Ann of Math.
                             {\bf 85} (1967), 351-358.

\bibitem {Kon} S. Kondo,      {\em Niemeier lattices, Mathieu groups, and
                               finite groups of symplectic automorphisms of
                               $K3$ surfaces}, with an appendix by S. Mukai,
                               Duke Math. J. {\bf 92} (1998), 593-603

\bibitem {Le} E. Lerman, {\em Symplectic cuts}, Math. Res. Lett.
                         {\bf 2} (1995), no. 3, 247-258.

\bibitem {Li} T.-J. Li, {\em Symplectic 4-manifolds with Kodaira dimension
                         zero}, J. Differential Geom. {\bf 74} (2006),
                           no. 2, 321-352.

\bibitem {MO} N. Machida and K. Oguiso, {\em On $K3$ surfaces
                        admitting finite non-symplectic group actions},
                         J. Math. Sci. Univ. Tokyo {\bf 5}(1998), 273-297.

\bibitem {McW} J.D. McCarthy and J.G. Wolfson, {\em Symplectic
                                resolution of isolated algebraic
                                singularities}, in Geometry,
                                topology, and dynamics (Montreal,
                                PQ, 1995), 101-105, CRM Proc.
                                Lecture Notes {\bf 15}, AMS, 1998


\bibitem {McC} M. McCooey, {\em Symmetry groups of four-manifolds},
                            Topology {\bf 41} (2002), no.4, 835-851.

\bibitem {MH} J.W. Morgan and H. Bass (Eds.), {\em The Smith conjecture},
              Pure and Applied Mathematics {\bf 112}, Academic Press,
              Orlando, FL, 1984.


\bibitem {MS} J.W. Morgan and Z. Szab\'{o}, {\em Homotopy $K3$ surfaces
                         and mod $2$ Seiberg-Witten invariants}, Math. Res.
                         Lett. {\bf 4} (1997), no.1, 17-21.

\bibitem {Mu} S. Mukai, {\em Finite groups of automorphisms of $K3$ surfaces
                         and the Mathieu group}, Invent. Math. {\bf 94} (1988),
                          183-221.

\bibitem {N} V. Nikulin, {\em Finite groups of automorphisms of K\"{a}hlerian
                            surfaces of type $K3$}, Trudy Mosk. Math. {\bf 38}
                            (1979), 75-137, Trans. Moscow Math. Soc. {\bf 38}
                            (1980), 71-135

\bibitem {Sct} P. Scott, {\em The geometries of $3$-manifolds},
                          Bull. London Math. Soc. {\bf 15}, 401-487 (1983)

\bibitem {Shi} I. Shimada, {\em On elliptic $K3$ surfaces}, Michigan
                            Math. J. {\bf 47}  (2000),  no. 3, 423--446.

\bibitem {T} C.H. Taubes, {\em $SW\Rightarrow Gr$: from the Seiberg-Witten
                          equations to pseudoholomorphic curves}, J. Amer.
                           Math. Soc. {\bf 9}(1996), 845-918, and
                           reprinted with errata in Proceedings of
                           the First IP Lectures Series, Volume II,
                           R. Wentworth ed., International Press,
                           Somerville, MA, 2000

\bibitem {Xiao} G. Xiao, {\em Galois covers between $K3$ surfaces},
                                  Ann. Inst. Fourier (Grenoble) {\bf 46}
                                  (1996), 73-88




\end{thebibliography}
\end{document}